\documentclass[10pt,reqno]{amsart}

\usepackage[all]{xy}
\usepackage{tom}
\usepackage{url}

\DeclareMathOperator{\IN}{in}
\DeclareMathOperator{\tin}{t-in}
\DeclareMathOperator{\Ass}{Ass}
\DeclareMathOperator{\minAss}{minAss}
\DeclareMathOperator{\val}{val}

\DeclareMathOperator{\lm}{lm}
\DeclareMathOperator{\lt}{lt}
\DeclareMathOperator{\lc}{lc}

\DeclareMathOperator{\Trop}{Trop}
\newcommand{\ux}{\underline{x}}

\newcommand{\longer}[1]{}
\newcommand{\short}[1]{#1}
\newcommand{\bmath}{\short{\begin{math}}\longer{\begin{displaymath}} }
\newcommand{\emath}{\short{\end{math}}\longer{\end{displaymath}} }

\begin{document}

   \parindent0cm

   \title[Lifting Points]{An algorithm for lifting points in a tropical variety}
   \author[A.\ Jensen, H.\ Markwig, T. Markwig]{Anders Nedergaard Jensen, Hannah Markwig and Thomas Markwig}
   \address {Anders Nedergaard Jensen, Institut f\"ur Mathematik, MA 4-5, Technische Universit\"at Berlin, 10623 Berlin, Germany }
   \email {jensen@math.tu-berlin.de}
   \address {Thomas Markwig, Fachbereich Mathematik, Technische Universit\"at
     Kaiserslautern, Postfach 3049, 67653 Kaiserslautern, Germany }
   \email {keilen@mathematik.uni-kl.de}
   \urladdr{http://www.mathematik.uni-kl.de/\textasciitilde keilen}
   \address{Hannah Markwig, Institute for Mathematics and its
     Applications, University of Minnesota, 400 Lind Hall 
     207 Church Street S.E., Minneapolis, MN 55455-0436 }
   \email{markwig@ima.umn.edu}
   \thanks{The first and third author would like to thank the
     Institute for Mathematics and its Applications (IMA) in Minneapolis for hospitality}
   
   \subjclass{Primary 13P10, 51M20, 16W60, 12J25; Secondary 14Q99, 14R99}

   \keywords{Tropical geometry, Puiseux series, Puiseux parametrisation.}
     
   \begin{abstract}
     The aim of this paper is to give a constructive proof of one of
     the basic theorems of tropical geometry: given a point on a
     tropical variety (defined using initial ideals), there exists a
     Puiseux-valued ``lift'' of this point in the algebraic
     variety. This theorem is so 
     fundamental because it justifies why a tropical variety (defined
     combinatorially using initial ideals) carries information about
     algebraic varieties: it is the image of an algebraic
     variety over the Puiseux series under the valuation map.  
     We have implemented the ``lifting algorithm'' using \textsc{Singular}
     and \texttt{Gfan} if the base field is $\Q$. As a byproduct we
     get an algorithm to compute 
     the Puiseux expansion of a space curve singularity in
     $(K^{n+1},0)$.
   \end{abstract}

   \maketitle


   \section{Introduction}

   In tropical geometry, algebraic varieties are replaced by certain
   piecewise linear objects called tropical varieties. Many algebraic
   geometry theorems have been ``translated'' to the tropical world
   (see for example \cite{Mi03},   \cite{Vig04}, \cite{SS04a},
   \cite{GM052} and many more). 
   Because new methods can be used in the tropical world --- for
   example, combinatorial methods --- and because the objects seem
   easier to deal with due to their piecewise linearity, tropical
   geometry is a promising tool for deriving new results in algebraic
   geometry. (For example, the Welschinger invariant can be computed tropically,
   see \cite{Mi03}). 

   There are two ways to define the tropical variety $\Trop(J)$ for an
   ideal $J$ in the polynomial ring $K\{\{t\}\}[x_1,\ldots,x_n]$ over the field of
   Puiseux series (see Definition
   \ref{def:puiseuxfield}). One way is to define the tropical variety
   combinatorially using $t$-initial ideals (see
   Definition~\ref{def:initial} 
   and Definition~\ref{def:tropicalvariety}, resp.\ \cite{SS04a}) --- this
   definition is more helpful when computing and it is the definition
   we use in this paper. 
   The other way to define tropical varieties is as the closure of
   the image of the algebraic variety $V(J)$ of $J$ in $K\{\{t\}\}^n$
   under the negative of 
   the valuation map (see Remark \ref{rem:puiseuxfield}, resp.\ \cite{RST03}, Definition 2.1) --- this gives more insight why
   tropical varieties carry information about algebraic varieties.

   It is our main aim in this paper to give a constructive proof that
   these two concepts coincide (see Theorem \ref{thm:2tropdef}),  and
   to derive that way an 
   algorithm which allows to lift a given point $\omega\in\Trop(J)$ to
   a point in $V(J)$ up to given order (see Algorithms \ref{alg:ZDL}
   and \ref{alg:RDZ}). The algorithm has been implemented
   using the commutative algebra system \textsc{Singular} (see
   \cite{GPS05}) and the programme \texttt{Gfan} (see \cite{gfan}), which
   computes Gr\"obner fans and tropical varieties. 

   Theorem~\ref{thm:2tropdef} has been proved in the case of a
   principal ideal by \cite{EKL04}, Theorem~2.1.1.  
   There is also a constructive proof for a principal ideal in
   \cite{Tab05}, Theorem~2.4. 
   For the general case, there is a proof in \cite{SS04}, Theorem~2.1,
   which has a gap however. Furthermore, there is a proof in
   \cite{Dra06}, Theorem~4.2, using affinoid algebras, and in
   \cite{Kat06}, Lemma~5.2.2, using flat schemes.
   A more general statement is proved in \cite{Pay07},
   Theorem~4.2. 
   Our proof has the
   advantage that it is constructive and works for an arbitrary ideal
   $J$. 

   We describe our algorithm first in the case where the ideal is
   $0$-dimensional. This algorithm can be viewed as a variant of an algorithm
   presented by Joseph Maurer in \cite{Mau80}, a paper from 1980. In
   fact, he uses the term ``critical tropism'' for a point in the tropical
   variety, even though tropical varieties were not defined by
   that time. Apparently, the notion goes back to Monique
   Lejeune-Jalabert and Bernard Teissier\footnote{Asked about this coincidence in the two
     notions Bernard Teissier sent us the following kind and interesting explanation:
     \emph{As far as I know the term did not exist before.
     We tried to convey the idea that giving different weights to some 
     variables made the space ''anisotropic'', and we were intrigued by the 
     structure, for example, of anisotropic projective spaces (which are 
     nowadays called weighted projective spaces).
     From there to ''tropismes critiques'' was  a quite natural linguistic 
     movement.
     Of course there was no ''tropical'' idea around, but as you say, it is 
     an amusing coincidence.
     The Greek ''Tropos'' usually designates change, so that ''tropisme 
     critique'' is well adapted to denote the values where the change of 
     weights becomes critical for the  
     computation of the initial ideal. The term ''Isotropic'', apparently due to 
     Cauchy, refers to the property of presenting the same (physical) 
     characters in all directions. Anisotropic is, of course, its negation.
     The name of Tropical geometry originates, as you probably know, from 
     tropical algebra which honours the Brazilian computer scientist
     Imre Simon living 
     close to the tropics, where the course of the sun changes back to the 
     equator.  In a way the tropics of Capricorn and Cancer represent, for 
     the sun, critical tropisms.}} (see
   \cite{LT73}).

   This paper is organised as follows: In Section~\ref{sec:basicnotation} we recall
   basic definitions and state the main result. In Section
   \ref{sec:zerodimensionalliftinglemma} we give a constructive proof
   of the main result in the $0$-dimensional case and deduce an
   algorithm. In Section~\ref{sec:arbitraryliftinglemma} we reduce the
   arbitrary case algorithmically to the $0$-dimensional case, and 
   in Section~\ref{sec:generalcommutativealgebra} we gather
   some simple results from commutative algebra for the lack of a
   better reference.  The
   proofs of both cases heavily rely on a good understanding of the
   relation of the dimension of an ideal $J$ over the Puiseux series with
   its $t$-initial ideal, respectively with its restriction to the
   rings $R_N[\ux]$ introduced below (see Definition
   \ref{def:puiseuxfield}). This will be studied in Section
   \ref{sec:dimension}. Some of the theoretical as well as the
   computational results use Theorem \ref{thm:stdtin} which was
   proved in \cite{Mar07} using standard bases in the mixed power
   series polynomial ring $K[[t]][\ux]$. We give an alternative proof
   in Section~\ref{sec:computinginitialideals}. 

   We would like to thank Bernd Sturmfels for suggesting the project and
   for many helpful discussions, and Michael Brickenstein, Gerhard Pfister and Hans
   Sch\"one\-mann for answering many questions concerning
   \textsc{Singular}. Also we would like to thank Sam Payne for
   helpful remarks and for pointing out a mistake in an earlier
   version of this paper. 
   
   Our programme can be downloaded from the web page
   \begin{center}
     www.mathematik.uni-kl.de/\textasciitilde keilen/en/tropical.html.
   \end{center}

   \section{Basic Notations and the Main Theorem}\label{sec:basicnotation}

   In this section we will introduce the basic notations used
   throughout the paper. 

   \begin{definition}\label{def:puiseuxfield}
     Let $K$ be an arbitrary field. We consider for $N\in \N_{>0}$ the discrete valuation ring
     \begin{displaymath}
       R_N=K\big[\big[t^\frac{1}{N}\big]\big]=
       \left\{\sum_{\alpha=0}^\infty a_\alpha\cdot t^\frac{\alpha}{N}\;\Big|\;a_\alpha\in K\right\}
     \end{displaymath}
     of formal power series in the unknown $t^\frac{1}{N}$
     with \emph{discrete valuation}
     \begin{displaymath}
       \val\left(\sum_{\alpha=0}^\infty a_\alpha\cdot
       t^\frac{\alpha}{N}\right)
       =\ord_t\left(\sum_{\alpha=0}^\infty a_\alpha\cdot
       t^\frac{\alpha}{N}\right)=
       \min\left\{\frac{\alpha}{N}\;\Big|\;a_\alpha\not=0\right\}\in\frac{1}{N}\cdot\Z\cup \{\infty\},
     \end{displaymath}
     and we denote by
     \bmath
         L_N=\Quot(R_N)
     \emath
     its quotient field. If $N\;|\;M$ then
     in an obvious way we can think of $R_N$ as a subring of $R_M$, and
     thus of $L_N$ as a subfield of $L_M$.
     We call the direct limit of the corresponding direct system
     \begin{displaymath}
       L=K\{\{t\}\}=\lim_{\longrightarrow} L_N=\bigcup_{N> 0}L_N
     \end{displaymath}
     the \emph{field of (formal) Puiseux series} over $K$. 
   \end{definition}
Recall that if $K$ is algebraically closed of characteristic $0$, then $L$ is algebraically closed.
   \begin{remark}\label{rem:puiseuxfield}
     If $0\not=N\in\N$ then
     \bmath
       S_N=\big\{1,t^\frac{1}{N},t^\frac{2}{N},t^\frac{3}{N},\ldots\big\}
     \emath
     is a multiplicatively closed subset of $R_N$, and obviously
     \begin{displaymath}
       L_N=S_N^{-1}R_N=\left\{t^\frac{-\alpha}{N}\cdot f\;\bigg|\;f\in
         R_N,\alpha\in\N\right\}\short{.}\longer{,} 
     \end{displaymath}
      \longer{since
      \begin{displaymath}
        R_N^*=\left\{\sum_{\alpha=0}^\infty a_\alpha\cdot t^\frac{\alpha}{N}\;\bigg|\;a_0\not=0\right\}.
      \end{displaymath}}
     The valuation of $R_N$ extends to $L_N$, and thus $L$, by
     \bmath
       \val\left(\frac{f}{g}\right)=\val(f)-\val(g)
     \emath
     for $f,g\in R_N$ with $g\not=0$. In particular, $\val(0)=\infty$.
   \end{remark}

   \begin{notation}
     Since an ideal $J\unlhd L[\ux]$ is generated by finitely many
     elements, the set
     \begin{displaymath}
       \mathcal{N}(J)=\big\{N\in\N_{>0}\;\big|\;\langle J\cap R_N[\ux]\rangle_{L[\ux]}=J\big\}
     \end{displaymath}
     is non-empty, and if $N\in\mathcal{N}(J)$ then
     $N\cdot \N_{>0}\subseteq\mathcal{N}(J)$. We also 
     introduce the notation $J_{R_{N}}=J\cap R_{N}[\ux]$. 
   \end{notation}


   \begin{remarkdefinition}\label{def:initial}
     Let $N\in\N_{>0}$, $w=(w_0,\ldots,w_n)\in
     \R_{<0}\times\R^n$, and $q\in\R$.

     We may consider the direct product
       \begin{displaymath}
         V_{q,w,N}=\prod_{\tiny\begin{array}{c}(\alpha,\beta)\in\N^{n+1}
             \\w\cdot(\frac{\alpha}{N},\beta)=q\end{array}}
         K\cdot t^\frac{\alpha}{N}\cdot \ux^\beta
       \end{displaymath}
       of $K$-vector spaces and its subspace
       \begin{displaymath}
         W_{q,w,N}=\bigoplus_{\tiny\begin{array}{c}(\alpha,\beta)\in\N^{n+1}
             \\w\cdot(\frac{\alpha}{N},\beta)=q\end{array}} K\cdot
         t^\frac{\alpha}{N}\cdot \ux^\beta.
     \end{displaymath}
     As a $K$-vector space the formal power series ring $K\big[\big[t^\frac{1}{N},\ux\big]\big]$ is just
     \begin{displaymath}
       K\big[\big[t^\frac{1}{N},\ux\big]\big]=\prod_{q\in\R}V_{q,w,N},
     \end{displaymath}
     and we can thus write any power series $f\in K\big[\big[t^\frac{1}{N},\ux\big]\big]$ in a
     unique way as
     \begin{displaymath}
       f=\sum_{q\in\R}f_{q,w} \;\;\;\mbox{ with }\;\;\;
       f_{q,w}\in V_{q,w,N}.
     \end{displaymath}   
     Note that this representation is independent of $N$ in the sense
     that if $f\in K\big[\big[t^\frac{1}{N'},\ux\big]\big]$ for some other
     $N'\in\N_{>0}$ then we get the same non-vanishing $f_{q,w}$ if
     we decompose $f$ with respect to $N'$. 
     
     Moreover, if $0\not=f\in R_N[\ux]\subset K\big[\big[t^\frac{1}{N},\ux\big]\big]$, then there is
     a \emph{maximal} $\hat{q}\in \R$ such that $f_{\hat{q},w}\not=0$ and 
     \bmath
       f_{q,w}\in W_{q,w,N}\longer{\;\;\;}\mbox{ for all
       }\longer{\;\;\;}q\in\R,
     \emath
     since the $\ux$-degree of the monomials involved in $f$ is
     bounded. We call the elements $f_{q,w}$ 
     \emph{$w$-quasihomogeneous} of $w$-degree $\deg_w(f_{q,w})=q\in \R$, 
     \begin{displaymath}
       \IN_w(f):=f_{\hat{q},w}\in K\big[t^\frac{1}{N},\ux\big]
     \end{displaymath}
     the \emph{$w$-initial form} of $f$, and 
     \begin{displaymath}
       \ord_w(f):=\hat{q}=\max\{\deg_w(f_{q,w})\;|\;f_{q,w}\not=0\}
     \end{displaymath}
     the \emph{$w$-order} of $f$. Set $\in_\omega(0)=0$. If $t^\beta
     x^\alpha\neq t^{\beta'}x^{\alpha'}$ are both monomials of
     $\IN_w(f)$, then $\alpha\neq \alpha'$. 

     For $I\subseteq R_N[\ux]$ we call
     \begin{displaymath}
       \IN_w(I)=\big\langle \IN_w(f)\;\big|\;f\in I\big\rangle
       \unlhd K\big[t^\frac{1}{N},\ux\big]
     \end{displaymath}
     the \emph{$w$-initial ideal} of $I$. Note that its definition
     depends on $N$.

     Moreover, we call for $f\in R_N[\ux]$
     \begin{displaymath}
       \tin_w(f)=\IN_w(f)(1,\ux)=\IN_w(f)_{|t=1}\in K[\ux]
     \end{displaymath}
     the \emph{$t$-initial form of $f$ w.r.t.\ $w$}, and if
     $f=t^\frac{-\alpha}{N}\cdot g\in L[\ux]$ with $g\in
     R_N[\ux]$ we set
     \begin{displaymath}
       \tin_w(f):=\tin_w(g).
     \end{displaymath}
     This definition does not depend on the particular representation of
     $f$\short{.}\longer{, since $t^\frac{-\alpha}{N}\cdot g=t^\frac{-\alpha'}{N'}\cdot
     g'$ implies that $t^\frac{\alpha'}{N'}\cdot
     g=t^\frac{\alpha}{N}\cdot g'$ in $R_{N\cdot N'}[\ux]$ and thus 
     \begin{displaymath}
       t^\frac{\alpha'}{N'}\cdot\IN_w(g)=\IN_w\big(t^\frac{\alpha'}{N'}\cdot
       g\big)=\IN_w\big(t^\frac{\alpha}{N}\cdot g'\big)=t^\frac{\alpha}{N}\cdot\IN_w(g'),
     \end{displaymath}
     which shows that $\tin_w(g)=\tin_w(g')$. }

     If $J\subseteq L[\ux]$ is a subset of $L[\ux]$, then 
     \begin{displaymath}
       \tin_w(J)=\langle \tin_w(f)\;|\;f\in J\rangle\unlhd K[\ux]
     \end{displaymath}
     is the \emph{$t$-initial ideal} of $J$, which does not depend on
     any $N$.

     For two $w$-quasihomogeneous elements $f_{q,w}\in W_{q,w,N}$ and
     $f_{q',w}\in W_{q',w,N}$ we have $f_{q,w}\cdot f_{q',w}\in
     W_{q+q',w,N}$. 
     In particular, 
     \bmath
       \IN_w(f\cdot g)=\IN_w(f)\cdot\IN_w(g)
     \emath
     for $f,g\in R_N[\ux]$, and 
     \bmath
       \tin_w(f\cdot g)=\tin_w(f)\cdot\tin_w(g)
     \emath     
     for $f,g\in L[\ux]$.
    
   \end{remarkdefinition}

   \begin{example}
     Let $w=(-1,-2,-1)$ and
     \begin{displaymath}
       f=\big(2t+t^{\frac{3}{2}}+t^2\big)\cdot
       x^2+(-3t^3+2t^4)\cdot y^2+ t^5xy^2 +\big(t+3t^2\big)\cdot
       x^7y^2.      
     \end{displaymath}
     Then $\ord_w(f)=-5$, $\IN_w(f)=2tx^2-3t^3y^2$, and 
     $\tin_w(f)= 2x^2-3y^2$. 
   \end{example}

   \begin{notation}\label{not:tin}
     Throughout this paper we will mostly use the weight $-1$ for the
     variable $t$, and in order to simplify the notation
     we will then usually write for $\omega\in\R^n$ 
     \begin{displaymath}
       \IN_\omega\;\;\;\mbox{ instead of }\;\;\;\IN_{(-1,\omega)}
     \end{displaymath}
     and
     \begin{displaymath}
       \tin_\omega\;\;\;\mbox{ instead of }\;\;\;\tin_{(-1,\omega)}.
     \end{displaymath}
     The case that $\omega=(0,\ldots,0)$ is of particular interest,
     and we will simply write
     \begin{displaymath}
       \IN_0\;\;\;\mbox{ respectively }\;\;\;\tin_{0}.
     \end{displaymath}
     This should not lead to any ambiguity.
   \end{notation}

   In general, the $t$-initial ideal of an ideal $J$ is not generated
   by the $t$-initial forms of the given generators of $J$. 

   \begin{example}\label{ex:tin}
     Let $J=\langle tx+y,x+t\rangle\lhd L[x,y]$ and
     $\omega=(1,-1)$. Then $y-t^2\in J$, but
     \begin{displaymath}
       y=\tin_\omega(y-t^2)\not\in \langle
       \tin_\omega(tx+y),\tin_\omega(x+t)\rangle=
       \langle x\rangle.
     \end{displaymath}
   \end{example}

   We can compute the $t$-initial ideal using standard bases by
   \cite{Mar07}, Corollary 6.11.

   \begin{theorem}\label{thm:stdtin}
     Let $J=\langle I\rangle_{L[\ux]}$
     with $I\unlhd K\big[t^\frac{1}{N},\ux\big]$, 
     $\omega\in\Q^n$ and $G$ be a standard basis of
     $I$ with respect to $>_\omega$ (see Remark \ref{rem:monomialordering}
     for the definition of $>_\omega$). 

     Then
     \bmath
       \tin_\omega(J)=\tin_\omega(I)=\big\langle\tin_\omega(G)\big\rangle\unlhd K[\ux].
     \emath
   \end{theorem}

   The proof of this theorem uses standard basis techniques in the
   ring $K[[t]][\ux]$. We give an alternative proof in Section
   \ref{sec:computinginitialideals}.

   \begin{example}
     In Example \ref{ex:tin}, $G=(tx+y,x+t,y-t^2)$ is a suitable
     standard basis and thus $\tin_\omega(J)=\langle x,y\rangle$. 
   \end{example}

   \begin{definition}\label{def:tropicalvariety}
     Let $J\unlhd L[\ux]$ be an ideal then the \emph{tropical variety}
     of $J$ is defined as
     \begin{displaymath}
       \Trop(J)=\{\omega\in\R^n\;|\;\tin_\omega(J)\mbox{ is monomial free}\}.
     \end{displaymath}
   \end{definition}

It is possible that $\Trop(J)=\emptyset$.

   \begin{example}
     Let $J=\langle x+y+1\rangle \subset L[x,y]$. As $J$ is
     generated by one polynomial $f$ which then automatically is a
     standard basis, the $t$-initial ideal
     $\tin_\omega(J)$ will be generated by $\tin_\omega(f)$ for any
     $\omega$. Hence
     $\tin_\omega(J)$ contains no monomial if and only if $\tin_\omega(f)$ is not
     a monomial. This is the case for all $\omega$ such that $\omega_1=\omega_2\geq
     0$, or $\omega_1=0\geq \omega_2$, or $\omega_2=0\geq \omega_1$. Hence the tropical
     variety $\Trop(J)$ looks as follows: 
     
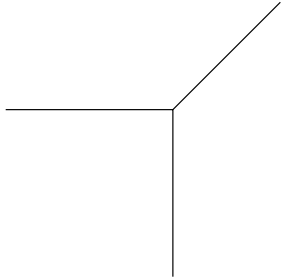
\begin{center}
       \begin{picture}(0,0)%
\includegraphics{tropline.pstex}%
\end{picture}%
\setlength{\unitlength}{3947sp}%
\begingroup\makeatletter\ifx\SetFigFont\undefined%
\gdef\SetFigFont#1#2#3#4#5{%
  \reset@font\fontsize{#1}{#2pt}%
  \fontfamily{#3}\fontseries{#4}\fontshape{#5}%
  \selectfont}%
\fi\endgroup%
\begin{picture}(1749,1749)(3889,-5023)
\end{picture}%

     \end{center}     
   \end{example}

   We need the following basic results about tropical varieties.

   \begin{lemma}\label{lem:tropicalvariety}
     Let $J,J_1,\ldots,J_k\unlhd L[\ux]$ be ideals. Then:
     \begin{enumerate}
     \item $J_1\subseteq J_2\;\;\;\Longrightarrow\;\;\;
       \Trop(J_1)\supseteq\Trop(J_2)$,
     \item $\Trop(J_1\cap\ldots\cap
       J_k)=\Trop(J_1)\cup\ldots\cup\Trop(J_k)$,
     \item
       $\Trop(J)=\Trop\big(\sqrt{J}\big)=\bigcup_{P\in\minAss(J)}\Trop(P)$, and
     \item $\Trop(J_1+J_2)\subseteq \Trop(J_1)\cap \Trop(J_2)$.  
     \end{enumerate}
   \end{lemma}
   \begin{proof}
     Suppose that $J_1\subseteq J_2$ and
     $\omega\in\Trop(J_2)\setminus\Trop(J_1)$. Then $\tin_\omega(J_1)$
     contains a monomial, but since $\tin_\omega(J_1)\subseteq
     \tin_\omega(J_2)$ this contradicts $\omega\in\Trop(J_2)$. Thus
     $\Trop(J_2)\subseteq\Trop(J_1)$. This shows (a).

     Since $J_1\cap \ldots\cap J_k\subseteq J_i$ for each $i=1,\ldots,
     k$ the first assertion implies that 
     \begin{displaymath}
       \Trop(J_1)\cup\ldots\cup\Trop(J_k)\subseteq\Trop(J_1\cap\ldots\cap J_k).
     \end{displaymath}
     Conversely, if $\omega\not\in\Trop(J_i)$ for $i=1,\ldots,k$
     then there exist polynomials $f_i\in J_i$ such that $\tin_\omega(f_i)$ is
     a monomial. But then $\tin_\omega(f_1\cdots
     f_k)=\tin_\omega(f_1)\cdots \tin_\omega(f_k)$ is a monomial and
     $f_1\cdots f_k\in J_1\cdots J_k\subseteq J_1\cap\ldots\cap J_k$. Thus
     $\omega\not\in \Trop(J_1\cap\ldots\cap J_k)$, which shows (b).

     For (c) it suffices to show that
     $\Trop(J)\subseteq\Trop\big(\sqrt{J}\big)$, since $J\subseteq
     \sqrt{J}=\bigcap_{P\in\minAss(J)}P$. If $\omega\not\in\Trop\big(\sqrt{J}\big)$
     then there is an $f\in\sqrt{J}$ such that $\tin_\omega(f)$ is a
     monomial and such that $f^m\in J$ for some $m$. But then
     $\tin_\omega(f^m)=\tin_\omega(f)^m$ is a monomial and thus
     $\omega\not\in\Trop(J)$. 

     Finally (d) is obvious from the definition.
   \end{proof}

   We are now able to state our main theorem.

   \begin{theorem}\label{thm:2tropdef}
     If $K$ is algebraically closed of characteristic zero and $J\unlhd K\{\{t\}\}[\ux]$ is
     an ideal then
     \begin{displaymath}
       \omega\in\Trop(J)\cap\Q^n
       \;\;\;\;\;\Longleftrightarrow\;\;\;\;\;
       \exists\;p\in V(J):\;-\val(p)=\omega\in\Q^n,
     \end{displaymath}
     where $\val$ is the coordinate-wise valuation.
   \end{theorem}

   The proof of one direction is straight forward and it does not
   require that $K$ is algebraically closed.

   \begin{proposition}\label{prop:tropical}
     If $J\unlhd L[\ux]$ is an ideal and $p\in V(J)\cap (L^{\ast})^n$, then $-\val(p)\in\Trop(J)$.
   \end{proposition}
   \begin{proof}
     Let $p=(p_1,\ldots,p_n)$, and let $\omega=-\val(p) \in
     \Q^n$. If $f\in J$, we have to show that $\tin_\omega(f)$ is not a
     monomial, but since this property is preserved when multiplying
     with some $t^\frac{\alpha}{N}$ we may as well assume that $f\in
     J_{R_N}$. As $p\in V(J)$, we know that $f(p)=0$. In particular
     the terms of lowest $t$-order in $f(p)$ have to cancel. But
     the terms of lowest order in $f(p)$ are
     \bmath
       \IN_\omega(f)(a_1\cdot t^{-\omega_1},\ldots,a_n\cdot t^{-\omega_n}),
     \emath
     where  $p_i=a_i\cdot t^{-\omega_i}+h.o.t.$.
     Hence $\IN_\omega(f)(a_1t^{-\omega_1},\ldots,a_nt^{-\omega_n})=0$,
     which is only possible if $\IN_\omega(f)$, and thus
     $\tin_\omega(f)$,  is not a monomial.      
   \end{proof}

Essentially, this was shown by Newton in \cite{New70}.
   \begin{remark}
     If the base field $K$ in Theorem \ref{thm:2tropdef} is not
     algebraically closed or not of characteristic zero, then the
     Puiseux series field is not algebraically closed (see e.g.\
     \cite{Ked01}). We therefore cannot expect to be able to lift each
     point in the tropical variety of an ideal $J\lhd K\{\{t\}\}[\ux]$
     to a point in $V(J)\subseteq K\{\{t\}\}^n$. However, if we
     replace $V(J)$ by the vanishing set, say $W$, of $J$ over the algebraic
     closure $\overline{L}$ of $K\{\{t\}\}$ then it is still true that each point $\omega$ in
     the tropical variety of $J$ can be lifted to a point $p\in W$
     such that $\val(p)=-\omega$. For this we note first that if
     $\dim(J)=0$ then the non-constructive proof of Theorem
     \ref{thm:liftinglemma} works by passing from $J$ to $\langle
     J\rangle_{\overline{L}[\ux]}$, taking into account that the
     non-archimedian valuation of a field in a natural way extends to
     its algebraic closure. And if $\dim(J)>0$ then we can add
     generators to $J$ by Proposition \ref{prop:intersect} and Remark
     \ref{rem:char} so as to 
     reduce to the zero dimensional case before passing to the
     algebraic closure of $K\{\{t\}\}$. 

     Note, it is even possible
     to apply Algorithm \ref{alg:ZDL} in the case of positive
     characteristic. However, due to the weird nature of the algebraic
     closure of the Puiseux series field in that case we cannot
     guarantee that the result will coincide with a
     solution of $J$ up to the order up to which it is computed. It
     may very well be the case that some intermediate terms are
     missing (see \cite{Ked01} Section 5).
   \end{remark}


   \section{Zero-Dimensional Lifting Lemma}\label{sec:zerodimensionalliftinglemma}

   In this section we want to give a constructive proof of the
   Lifting Lemma \ref{thm:liftinglemma}. 

   \begin{theorem}[Lifting Lemma]\label{thm:liftinglemma}
     Let $K$ be an algebraically closed field of
     characteristic zero and $L=K\{\{t\}\}$.
     If $J\lhd L[\ux]$ is a zero dimensional ideal and
     $\omega\in\Trop(J)\cap\Q^n$, then there is a point $p\in V(J)$
     such that $-\val(p)=\omega$.
   \end{theorem}
   \begin{proof}[Non-Constructive Proof]
     If $\omega\in\Trop(J)$ then by Lemma \ref{lem:tropicalvariety} there is an associated prime
     $P\in\minAss(J)$ such that $\omega\in\Trop(P)$. But since
     $\dim(J)=0$ the ideal $P$ is necessarily a maximal ideal, and
     since $L$ is algebraically closed it is of the form
     \begin{displaymath}
       P=\langle x_1-p_1,\ldots,x_n-p_n\rangle
     \end{displaymath}
     with $p_1,\ldots,p_n\in L$. Since $\omega\in\Trop(P)$ the ideal
     $\tin_\omega(P)$ does not contain any monomial, and therefore
     necessarily
     \bmath
       \ord_t(p_i)=-\omega_i
     \emath
     for all $i=1,\ldots,n$. This shows that $p=(p_1,\ldots,p_n)\in
     V(P)\subseteq V(J)$ and $\val(p)=-\omega$.
   \end{proof}

     The drawback of this proof is that in order to find $p$ one would
     have to be able to find the associated primes of $J$ which would
     amount to something close to primary decomposition over $L$. This
     is of course not feasible. We will instead adapt
     the constructive proof that $L$ is algebraically closed, i.e.\ the
     Newton-Puiseux Algorithm for plane curves, which has already been
     generalised to space curves (see \cite{Mau80}, \cite{AMNR92}) to our
     situation in order to compute the point $p$ up to any given order. 
     
     The idea behind this is very simple and the first recursion step
     was basically already explained in the proof of Proposition
     \ref{prop:tropical}. Suppose we have a polynomial $f\in R_N[\ux]$
     and a point 
     \begin{displaymath}
       p=\left(u_1\cdot t^{\alpha_1}+v_1\cdot t^{\beta_1}+\ldots,\ldots,
           u_n\cdot t^{\alpha_n}+v_n\cdot
           t^{\beta_n}+\ldots\right)\in (L^\ast)^n.
     \end{displaymath}
     Then, a priori, the  term of lowest $t$-order in $f(p)$ will be
     \bmath
       \IN_{-\alpha}(f)(u_1\cdot t^{\alpha_1},\ldots,u_n\cdot t^{\alpha_n}).
     \emath
     Thus, in order for $f(p)$ to be zero it is necessary that
     \bmath
       \tin_{-\alpha}(f)(u_1,\ldots,u_n)=0.
     \emath
     Let $p'$ denote the tail of $p$, that is $p_i= u_i\cdot
     t^{\alpha_i}+t^{\alpha_i}\cdot p_i'$. 
     Then $p'$ is a zero of 
     \begin{displaymath}
       f'=f\big(t^{\alpha_1}\cdot(u_1+x_1),\ldots,t^{\alpha_n}\cdot(u_n+x_n)\big).
     \end{displaymath}
     The same arguments then show that 
     \bmath
       \tin_{\alpha-\beta}(f')(v_1,\ldots,v_n)=0,
     \emath
     and assuming now that none of the $v_i$ is zero we find
     $\tin_{\alpha-\beta}(f')$ must be monomial free, that is
     $\alpha-\beta$ is a point in the tropical variety and all its
     components are strictly negative. 

     The basic idea for the algorithm which computes a suitable $p$ is
     thus straight forward. Given $\omega=-\alpha$ in the tropical variety of
     an ideal $J$, compute a point $u\in V(\tin_\omega(J))$ apply the
     above transformation to $J$ and compute a negative-valued point in the tropical
     variety of the transformed ideal. Then go on recursively.

     It
     may happen that the solution that we are about to construct this
     way has some component with only finitely many terms. Then after
     a finite number of steps there might be no suitable $\omega$ in the
     tropical variety. However, in that situation we can  simply
     eliminate the corresponding variable for the further
     computations.

   \begin{example}
     Consider the ideal $J=\langle f_1,\ldots,f_4\rangle\lhd L[x,y]$ with
     \begin{displaymath}
       \begin{array}{ll}
         f_1=&y^2+4t^2y+(-t^3+2t^4-t^5),\\
         f_2=&(1+t)\cdot x-y+(-t-3t^2),\\
         f_3=&xy+(-t+t^2)\cdot x+(t^2-t^4),\\
         f_4=&x^2-2tx+(t^2-t^3).
       \end{array}
     \end{displaymath}
     The $t$-initial ideal of $J$ with respect to $\omega=\big(-1,-\frac{3}{2}\big)$ is
     \begin{displaymath}
       \tin_\omega(J)=\langle y^2-1,x-1\rangle,
     \end{displaymath}
     so that $\omega\in\Trop(J)$ and  $u=(1,1)$ is a suitable choice. Applying the
     transformation $\gamma_{\omega,u}:(x,y)\mapsto\big(t\cdot (1+x),t^\frac{3}{2}\cdot
     (1+y)\big)$ to $J$ we get $J'=\langle f_1',\ldots,f_4'\rangle$
     with
     \begin{displaymath}
       \begin{array}{ll}
         f_1'&=t^3y^2+\big(2t^3+4t^\frac{7}{2}\big)\cdot y+\big(4t^\frac{7}{2}+2t^4-t^5\big),\\
         f_2'&=(t+t^2)\cdot x-t^\frac{3}{2}\cdot y+\big(-t^\frac{3}{2}-2t^2\big),\\
         f_3'&=t^\frac{5}{2}\cdot xy+\big(-t^2+t^3+t^\frac{5}{2}\big)\cdot x
         +t^\frac{5}{2}\cdot y+\big(t^\frac{5}{2}+t^3-t^4\big),\\
         f_4'&=t^2x^2-t^3.
       \end{array}       
     \end{displaymath}
     This shows that the $x$-coordinate of a solution of $J'$
     necessarily is $x=\pm t^\frac{1}{2}$, and we could substitute
     this for $x$ in the other equations in order to reduce by one
     variable. We will instead see
     what happens when we go on with our algorithm.

     The $t$-initial ideal of $J'$ with respect to
     $\omega'=\big(-\frac{1}{2},-\frac{1}{2}\big)$ is
     \begin{displaymath}
       \tin_{\omega'}(J')=\langle y+2,x-1\rangle,
     \end{displaymath}
     so that $\omega'\in\Trop(J')$ and $u'=(1,-2)$ is our only choice. Applying the
     transformation $\gamma_{\omega',u'}:(x,y)\mapsto\big(t^\frac{1}{2}\cdot (1+x),t^\frac{1}{2}\cdot
     (-2+y)\big)$ to $J'$ we get the ideal $J''=\langle
     f_1'',\ldots,f_4''\rangle$ with 
     \begin{displaymath}
       \begin{array}{ll}
         f_1''&=t^4y^2+2t^\frac{7}{2}y+\big(-2t^4-t^5\big),\\
         f_2''&=\big(t^\frac{3}{2}+t^\frac{5}{2}\big)\cdot x
         -t^2\cdot y+t^\frac{5}{2},\\
         f_3''&=t^\frac{7}{2}\cdot xy+\big(-t^\frac{5}{2}+t^3-t^\frac{7}{2}\big)\cdot x
         +\big(t^3+t^\frac{7}{2}\big)\cdot y+\big(-t^\frac{7}{2}-t^4\big),\\
         f_4''&=t^3x^2+2t^3x.
       \end{array}       
     \end{displaymath}
     If we are to find an $\omega''\in\Trop(J'')$, then $f_4''$ implies
     that necessarily $\omega_1''=0$. But we are looking for an
     $\omega''$ all of whose entries are strictly negative. The reason
     why this does not exist is that there is a solution of $J''$ with
     $x=0$. We thus have to eliminate the variable $x$, and replace
     $J''$ by the ideal $J'''=\langle f'''\rangle$ with
     \begin{displaymath}
       f'''=y-t^\frac{1}{2}.
     \end{displaymath}
     Then $\omega'''=-\frac{1}{2}\in\Trop(J''')$ and
     $\tin_{\omega'''}(f''')=y-1$. Thus $u'''=1$ is our only choice,
     and since $f'''(u'''\cdot
     t^{-\omega'''})=f'''(t^\frac{1}{2})=0$ we are done. 

     Backwards substitution gives
     \begin{align*}
       p=&\left(t^{\omega_1}\cdot \left(u_1+t^{\omega_1'}\cdot\left(u_1'+0\right)\right),
         t^{\omega_2}\cdot\left(u_2+t^{\omega_2'}\cdot\left(u_2'+t^{\omega_2'''}
             \cdot u'''\right)\right)\right)\\
       =&\left(t\cdot \left(1+t^\frac{1}{2}\right),
         t^\frac{3}{2}\cdot\left(1+t^\frac{1}{2}\cdot\left(-2+t^\frac{1}{2}
             \right)\right)\right)\\
       =&\left(t+t^\frac{3}{2},t^\frac{3}{2}-2t^2+t^\frac{5}{2}\right)
     \end{align*}
     as a point in $V(J)$ with
     $\val(p)=\big(1,\frac{3}{2}\big)=-\omega$. Note that in general
     the procedure will not terminate.
   \end{example}

   For the proof that this algorithm works we need two types of
   transformations which we are now going to introduce and study.

   \begin{definitionremark}\label{rem:liftinglemma}
     For $\omega'\in\Q^n$ let us consider the $L$-algebra
     isomorphism 
     \begin{displaymath}
       \Phi_{\omega'}:L[\ux]\longrightarrow L[\ux]:
       x_i\mapsto t^{-\omega'_i}\cdot x_i,
     \end{displaymath}
     and the isomorphism which it induces on $L^n$
     \begin{displaymath}
       \phi_{\omega'}:L^n\rightarrow L^n:
       (p_1',\ldots,p_n')\mapsto 
       \big(t^{-\omega'_1}\cdot p_1',\ldots,t^{-\omega'_n}\cdot p_n'\big).
     \end{displaymath}
     Suppose we have found a $p'\in V\big(\Phi_{\omega'}(J)\big)$,
     then $p=\phi_{\omega'}(p')\in V(J)$ and
     $\val(p)=\val(p')-\omega'$. 

     Thus choosing $\omega'$
     appropriately we may in Theorem \ref{thm:liftinglemma}
     assume that $\omega\in\Q_{< 0}^n$, which due to Corollary
     \ref{cor:dimension:D} implies that the dimension of $J$ behaves
     well when contracting to the power series ring $R_N[\ux]$ for a
     suitable $N$. 

     Note also the following properties of $\Phi_{\omega'}$, which we
     will refer to quite frequently. 
     If $J\unlhd L[\ux]$ is an ideal, then
     \begin{displaymath}
       \dim(J)=\dim\big(\Phi_{\omega'}(J)\big)\;\mbox{ and }\;
       \tin_{\omega'}(J)=\tin_0\big(\Phi_{\omega'}(J)\big),
     \end{displaymath}
     where the latter is due to the fact that 
     \begin{displaymath}
       \deg_w\big(t^\alpha\cdot\ux^\beta\big)
       =-\alpha+\omega'\cdot\beta=
       \deg_v\big(t^{\alpha-\omega'\cdot\beta}\cdot\ux^\beta\big)= 
       \deg_v\big(\Phi_{\omega'}(t^\alpha\cdot \ux^\beta)\big)
     \end{displaymath}   
     with $w=(-1,\omega')$ and  $v=(-1,0,\ldots,0)$.
     
     \longer{Moreover, since $\Phi_{\omega'}$ is an isomorphism
     \begin{displaymath}
       \Phi_{\omega'}\big(\Ass(J)\big)=\Ass\big(\Phi_{\omega'}(J)\big).
     \end{displaymath}}
   \end{definitionremark}

   \begin{definitionremark}\label{def:transformation}
     For  $u=(u_1,\ldots,u_n)\in K^n$, $\omega\in\Q^n$ and $w=(-1,\omega)$ we consider
     the $L$-algebra isomorphism 
     \begin{displaymath}
       \gamma_{\omega,u}:L[\ux]\longrightarrow L[\ux]:
       x_i\mapsto t^{-\omega_i}\cdot (u_i+x_i),
     \end{displaymath}
     and its effect on a $w$-quasihomogeneous element
     \begin{displaymath}
       f_{q,w}=
       \sum_{\tiny\begin{array}{c}(\alpha,\beta)\in\N^{n+1}\\
           -\frac{\alpha}{N}+\omega\cdot\beta=q\end{array}} 
       a_{\alpha,\beta}\cdot t^\frac{\alpha}{N}\cdot \ux^\beta.
     \end{displaymath}
     If we set
     \begin{displaymath}
       p_\beta:=\prod_{i=1}^n(u_i+x_i)^{\beta_i}-u^\beta\in\langle x_1,\ldots,x_n\rangle
       \lhd K[\ux]      
     \end{displaymath}
     then 
     \renewcommand{\arraystretch}{1.3}
     \begin{equation}\label{eq:transformation:1}
        \begin{array}{rcl}
        \gamma_{\omega,u}(f_{q,w})&=& \sum\limits_{-\frac{\alpha}{N}+\omega\cdot\beta=q}
        a_{\alpha,\beta} \cdot t^\frac{\alpha}{N}\cdot \prod\limits_{i=1}^n t^{-\omega_i\cdot
         \beta_i}\cdot (u_i+x_i)^{\beta_i}\\
       &= & \;t^{-q}\cdot \sum\limits_{-\frac{\alpha}{N}+\omega\cdot\beta=q} a_{\alpha,\beta}\cdot
       \big(u^\beta+p_\beta)\\
       &= & \;t^{-q}\cdot\bigg(f_{q,w}(1,u)+ \sum\limits_{-\frac{\alpha}{N}+\omega\cdot\beta=q}
         a_{\alpha,\beta}\cdot p_\beta\bigg)\\
       &= & \; t^{-q}\cdot f_{q,w}(1,u)+t^{-q}\cdot p_{f_{q,w},u},
        \end{array}
     \end{equation}    
     \renewcommand{\arraystretch}{1}
     with    
     \begin{displaymath}
       p_{f_{q,w},u}:=\sum_{-\frac{\alpha}{N}+w\cdot\beta=q} a_{\alpha,\beta}\cdot
       p_\beta\in \langle x_1,\ldots,x_n\rangle
       \lhd K[\ux].
     \end{displaymath}     
     In particular, if $\omega\in\frac{1}{N}\cdot\Z^n$ and
     $f=\sum_{q\leq\hat{q}}f_{q,w}\in R_N[\ux]$ with $\hat{q}=\ord_\omega(f)$ then
     \begin{displaymath}
       \gamma_{\omega,u}(f)=t^{-\hat{q}}\cdot g
     \end{displaymath}
     where
     \begin{displaymath}
       g=\sum_{q\leq \hat{q}}\big(t^{\hat{q}-q}\cdot
       f_{q,w}(1,u)+t^{\hat{q}-q}\cdot p_{f_{q,w},u}\big)\in R_N[\ux].
     \end{displaymath}
     \hfill$\Box$
   \end{definitionremark}

   The following lemma shows that if we consider the transformed ideal
   $\gamma_{\omega,u}(J)\cap R_N[\ux]$ in the power series ring
   $K\big[\big[t^\frac{1}{N},\ux\big]\big]$ then it defines the germ
   of a space curve through the origin. This allows us then in
   Corollary \ref{cor:tropnegative} to apply normalisation to find a negative-valued
   point in the tropical variety of $\gamma_{\omega,u}(J)$.

   \begin{lemma}\label{lem:transformation}     
     Let $J\lhd L[\ux]$, let $\omega\in
     \Trop(J)\cap\frac{1}{N}\cdot \Z^n$, and 
     $u\in V\big(\tin_\omega(J)\big)\subset K^n$. Then
     \begin{displaymath}
       \gamma_{\omega,u}(J)\cap R_N[\ux]\subseteq \big\langle
       t^\frac{1}{N},x_1,\ldots,x_n\big\rangle \lhd R_N[\ux].
     \end{displaymath}
   \end{lemma}   
   \begin{proof}
     Let $w=(-1,\omega)$ and
     $0\not=f=\gamma_{\omega,u}(h)\in\gamma_{\omega,u}(J)\cap
     R_N[\ux]$ with $h\in 
     J$. Since $f$ is a polynomial in $\ux$ we have
     \begin{displaymath}
       h=\gamma_{\omega,u}^{-1}(f)=f(t^{\omega_1}\cdot x_1-u_1,\ldots,t^{\omega_n}\cdot x_n-u_n)
       \in t^m\cdot R_N[\ux]
     \end{displaymath}
     for some $m\in\frac{1}{N}\cdot\Z$. We can thus decompose $g:=t^{-m}\cdot
     h\in J_{R_N}$ into its $w$-quasihomogeneous parts, say
     \begin{displaymath}
       t^{-m}\cdot h=g=\sum_{q\leq \hat{q}} g_{q,w},
     \end{displaymath}
     where $\hat{q}=\ord_\omega(g)$ and thus
     $g_{\hat{q},w}=\IN_\omega(g)$ is the $w$-initial form of $g$. As 
     we have seen in Remark \ref{def:transformation} there are polynomials
     $p_{g_{q,w},u}\in\langle x_1,\ldots,x_n\rangle\lhd K[\ux]$ such that
     \begin{displaymath}
       \gamma_{\omega,u}(g_{q,w})=t^{-q}\cdot g_{q,w}(1,u)+t^{-q}\cdot p_{g_{q,w},u}.
     \end{displaymath}
     But then
     \begin{align*}
       f=&\;\gamma_{\omega,u}(h)=\gamma_{\omega,u}(t^m\cdot
       g)=t^m\cdot \gamma_{\omega,u}(g)=
       t^m\cdot\gamma_{\omega,u}\left(\sum_{q\leq \hat{q}}
       g_{q,\omega}\right)\\ 
       =&\;t^m\cdot\sum_{q\leq \hat{q}} \big(t^{-q}\cdot g_{q,w}(1,u)+t^{-q} \cdot p_{g_{q,w},u}\big)\\
       =&\;t^{m-\hat{q}}\cdot g_{\hat{q},w}(1,u)+t^{m-\hat{q}}\cdot
       p_{g_{\hat{q},w},u}+\sum_{q<\hat{q}} t^{m-q}\cdot
       \big(g_{q,w}(1,u)+p_{g_{q,w},u}\big).
     \end{align*}
     However, since $g\in J$ and $u\in V\big(\tin_\omega(J)\big)$ we have 
     \begin{displaymath}
       g_{\hat{q},w}(1,u)=\tin_\omega(g)(u)=0
     \end{displaymath}
     and thus using \eqref{eq:transformation:1} we get
     \begin{displaymath}
       p_{g_{\hat{q},w},u}=t^{\hat{q}}\cdot
       \left(\gamma_{\omega,u}(g_{\hat{q},w})-t^{-\hat{q}}\cdot
         g_{\hat{q},w}(1,u)\right)
       =t^{\hat{q}}\cdot \gamma_{\omega,u}(g_{\hat{q},w})\not=0,
     \end{displaymath}
     since $g_{\hat{q},w}=\IN_\omega(g)\not=0$ and $\gamma_{\omega,u}$ is an
     isomorphism. We see in particular, that $m-\hat{q}\geq 0$  since
     $f\in R_N[\ux]$ and $p_{g_{\hat{q},w},u}\in\langle x_1,\ldots,x_n\rangle\lhd K[\ux]$, and hence
     \begin{displaymath}
       f=t^{m-\hat{q}}\cdot p_{g_{\hat{q},w},u}+\sum_{q<\hat{q}} t^{m-q}\cdot
       \big(g_{q,w}(1,u)+p_{g_{q,w},u}\big)\in
       \big\langle t^\frac{1}{N},x_1,\ldots,x_n\big\rangle.
     \end{displaymath}
   \end{proof}

   The following corollary assures the existence of a negative-valued point in the
   tropical variety of the transformed ideal -- after possibly eliminating those variables
   for which the components of the solution will be zero.

   \begin{corollary}\label{cor:tropnegative}
     Suppose that $K$ is an algebraically closed field of characteristic zero.
     Let $J\lhd L[\ux]$ be a zero-dimensional ideal, let $\omega\in
     \Trop(J)\cap\Q^n$, and 
     $u\in V\big(\tin_\omega(J)\big)\subset K^n$. Then
     \begin{displaymath}
       \exists\;p=(p_1,\ldots,p_n)\in
       V\big(\gamma_{\omega,u}(J)\big)\;:\;
      \forall i:\; \val(p_i)\in\Q_{>0}\cup\{\infty\}.
     \end{displaymath}
     In particular, if $n_p=\#\{p_i\;|\;p_i\not=0\}>0$ and
     $\ux_p=(x_i\;|\;p_i\not=0)$, then
     \begin{displaymath}
       \Trop\big(\gamma_{\omega,u}(J)\cap L[\ux_p]\big)\cap \Q_{<0}^{n_p}\not=\emptyset.
     \end{displaymath}     
   \end{corollary}
   \begin{proof}
     We may choose an $N\in \mathcal{N}(\gamma_{\omega,u}(J))$ and such that $\omega\in
     \frac{1}{N}\cdot \Z_{\leq 0}^n$. Let $I=\gamma_{\omega,u}(J)\cap R_N[\ux]$.

     Since $\gamma_{\omega,u}$ is an isomorphism we know that 
     \begin{displaymath}
       0=\dim(J)=\dim\big(\gamma_{\omega,u}(J)\big),
     \end{displaymath}
     and by Proposition \ref{prop:pd} we know that 
     \begin{displaymath}
       \Ass(I)=\big\{P_{R_N}\;\big|\;P\in\Ass\big(\gamma_{\omega,u}(J)\big)\big\}.
     \end{displaymath}
     Since the maximal ideal 
     \begin{displaymath}
       \m=\big\langle
       t^\frac{1}{N},x_1,\ldots,x_n\big\rangle_{R_N[\ux]}\lhd R_N[\ux]
     \end{displaymath}
     contains the element $t^\frac{1}{N}$, which is a unit in $L[\ux]$,
     it cannot be the contraction of a prime ideal in $L[\ux]$. In
     particular, $\m\not\in\Ass(I)$. Thus there must be a
     $P\in\Ass(I)$ such that
     \bmath
       P\subsetneqq \m,
     \emath
     since by Lemma \ref{lem:transformation} $I\subset \m$ and since
     otherwise $\m$ would be minimal over $I$ and hence associated to
     $I$.

     The strict inclusion  implies that
     \bmath
       \dim(P)\geq 1,
     \emath
     while Theorem \ref{thm:dimension:A} shows that
     \begin{displaymath}
       \dim(P)\leq\dim(I)\leq
       \dim\big(\gamma_{\omega,u}(J)\big) +1=1.
     \end{displaymath}
     Hence the ideal $P$ is a $1$-dimensional prime ideal in
     \bmath
       R_N[\ux]\subset K\big[\big[t^\frac{1}{N},\ux\big]\big],
     \emath 
     where the latter is the completion of the former with respect to
     $\m$. Since $P\subset \m$, the completion $\hat{P}$ of
     $P$ with respect to $\m$ is also $1$-dimensional \longer{(see e.g.\
     \cite{AM69} Cor.\ 11.19). By \cite{GR71} Satz II.7.2}\short{and } the
     normalisation 
     \begin{displaymath}
       \psi:K\big[\big[t^\frac{1}{N},\ux\big]\big]/\hat{P}\hookrightarrow \widetilde{R}\simeq K[[s]]
     \end{displaymath}
     \short{gives a parametrisation where we may assume that
       $\psi\big(t^\frac{1}{N}\big)=s^M$ for some $M\in \N_{>0}$ since $K$ is algebraically
       closed and of characteristic zero (see e.g.\ \cite{DP00} Cor.\ 4.4.10 for $K=\C$).}
     Let now $s_i=\psi(x_i)\in K[[s]]$
     then necessarily
     \bmath
       a_i=\ord_s(s_i)>0,
     \emath
     since $\psi$ is a local $K$-algebra homomorphism, and 
     \bmath
       f(s^M,s_1,\ldots,s_n)=\psi(f)=0
     \emath
     for all $f\in \hat{P}$.
     Taking $I\subseteq P\subset \hat{P}$ and $\gamma_{\omega,u}(J)=\langle I\rangle$
     into account and replacing $s$ by $t^\frac{1}{N\cdot M}$ we get
     \begin{displaymath}
       f\big(t^\frac{1}{N},p)=0 \;\;\;\mbox{ for all }\;f\in \gamma_{\omega,u}(J)
     \end{displaymath}
     where 
     \begin{displaymath}
       p=\Big(s_1\big(t^\frac{1}{N\cdot M}\big),\ldots,s_n\big(t^\frac{1}{N\cdot M}\big)\Big)
       \in R_{N\cdot M}^n\subseteq L^n.
     \end{displaymath}
     Moreover,
     \begin{displaymath}
       \val(p_i)=\frac{a_i}{N\cdot M}\in\Q_{>0}\cup\{\infty\},
     \end{displaymath}
     and every $f\in\gamma_{\omega,u}(J)\cap L[\ux_p]$ vanishes at the
     point $p'=(p_i\;|\;p_i\not=0)$. By
     Proposition \ref{prop:tropical} 
     \begin{displaymath}
       -\val(p')\in\Trop\big(\gamma_{\omega,u}(J)\cap L[\ux_p]\big)\cap \Q_{<0}^{n_p}.
     \end{displaymath}
   \end{proof}

   \begin{proof}[Constructive Proof of Theorem \ref{thm:liftinglemma}]
     Recall that by Remark \ref{rem:liftinglemma} we may assume that
     $\omega\in\Q_{<0}^n$. 
     It is our first aim to construct recursively sequences of the
     following objects for $\nu\in\N$:
     \begin{itemize}
     \item natural numbers $1\leq n_\nu\leq n$,
     \item natural numbers $1\leq i_{\nu,1}<\ldots<i_{\nu,n_\nu}\leq n$,
     \item subsets of variables $\ux_\nu=(x_{i_{\nu,1}},\ldots,x_{i_{\nu,n_\nu}})$,
     \item ideals $J_\nu'\lhd L[\ux_{\nu-1}]$,
     \item ideals $J_\nu\lhd L[\ux_\nu]$,
     \item vectors 
       $\omega_\nu=(\omega_{\nu,i_{\nu,1}},\ldots,\omega_{\nu,i_{\nu,n_\nu}})
       \in\Trop(J_\nu)\cap(\Q_{<0})^{n_\nu}$, and
     \item vectors $u_\nu=(u_{\nu,i_{\nu,1}},\ldots,u_{\nu,i_{\nu,n_\nu}})\in
     V\big(\tin_{\omega_\nu}(J_\nu)\big)\cap (K^*)^{n_\nu}$. 
     \end{itemize}

     We set $n_0=n$, $\ux_{-1}=\ux_0=\ux$, $J_0=J_0'=J$, and $\omega_0=\omega$, and since $\tin_\omega(J)$
     is monomial free by assumption and $K$ is algebraically closed we
     may choose a $u_0\in 
     V\big(\tin_{\omega_0}(J_0)\big)\cap (K^*)^{n_0}$. We then define
     recursively for $\nu\geq 1$
     \begin{displaymath}
       J_\nu'=\gamma_{\omega_{\nu-1},u_{\nu-1}}(J_{\nu-1}).
     \end{displaymath}
     By Corollary \ref{cor:tropnegative} we may choose a
     point $q\in V(J_\nu')\subset L^{n_{\nu-1}}$ such that
     $\val(q_i)=\ord_t(q_i)>0$ for all $i=1,\ldots,n_{\nu-1}$. As in Corollary
     \ref{cor:tropnegative} we set
     \begin{displaymath}
       n_\nu=\#\{q_i\;|\;q_i\not=0\}\in\{0,\ldots, n_{\nu-1}\},
     \end{displaymath}
     and we denote by 
     \begin{displaymath}
       1\leq i_{\nu,1}<\ldots<i_{\nu,n_\nu}\leq n
     \end{displaymath}
     the indexes $i$ such that $q_i\not=0$.
     
     If $n_\nu=0$ we simply stop the process, while if $n_\nu\not=0$ we set
     \begin{displaymath}
       \ux_{\nu}=(x_{i_{\nu,1}},\ldots,x_{i_{\nu,n_\nu}})\subseteq\ux_{\nu-1}.
     \end{displaymath}
     We then set
     \begin{displaymath}
       J_\nu=\big(J_\nu'+\langle \ux_{\nu-1}\setminus\ux_\nu\rangle\big)\cap L[\ux_\nu],
     \end{displaymath}
     and by Corollary \ref{cor:tropnegative} we can choose      
     \begin{displaymath}
       \omega_\nu=(\omega_{\nu,i_{\nu,1}},\ldots,\omega_{\nu,i_{\nu,n_\nu}})
       \in\Trop(J_\nu)\cap\Q_{<0}^{n_\nu}.
     \end{displaymath}
     Then $\tin_{\omega_\nu}(J_\nu)$ is monomial free, so that we
     can choose a
     \begin{displaymath}
       u_\nu=(u_{\nu,i_{\nu,1}},\ldots,u_{\nu,i_{\nu,n_\nu}})\in
       V\big(\tin_{\omega_\nu}(J_\nu)\big)\cap (K^*)^{n_\nu}.
     \end{displaymath}
     Next we define
     \begin{displaymath}
       \varepsilon_i=
       \sup\big\{\nu\;\big|\;i\in\{i_{\nu,1},\ldots,i_{\nu,n_\nu}\}\big\}\in\N\cup\{\infty\} \mbox{ and }
     \end{displaymath}
      \begin{displaymath}
       p_{\mu,i}=\sum_{\nu=0}^{\min\{\varepsilon_i,\mu\}} 
       u_{\nu,i}\cdot t^{-\sum_{j=0}^\nu\omega_{j,i}}
     \end{displaymath}
     for $i=1,\ldots,n$. \short{All $\omega_{\nu,i}$ are strictly
     negative, which is necessary to see that the $p_{\mu,i}$ converge
     to a Puiseux series.}\longer{Since
     all $\omega_{\nu,i}$ are  strictly 
     negative, we can \emph{hope} to show that for $\mu\mapsto\infty$ 
     the $p_{\mu,i}$ converge in some $R_N$ with respect to the
     $\big\langle t^\frac{1}{N}\big\rangle$-adic topology. For this we only have
     to show that for each $i=1,\ldots,k$ there is some $N_i$ such that 
     \begin{equation}\label{eq:lifting:2}
       \{\omega_{\nu,i}\;|\;\nu=0,\ldots,\varepsilon_i\}\subset \frac{1}{N_i}\cdot\Z.
     \end{equation}
     If $\varepsilon_i<\infty$ this is obvious, and we may thus assume
     that $\varepsilon_i=\infty$.
   }
     Note that in the case $n=1$ the described procedure is just the classical
     Puiseux expansion (see e.g.\ \cite{DP00} Thm.\ 5.1.1 for the case
     $K=\C$). 
     \short{To see that the $p_{\mu,i}$ converge to a Puiseux
       series (i.e.\ that there exists a common denominator
     $N$ for the exponents as $\mu$ goes to infinity), the general case can easily be reduced to the case
       $n=1$ by projecting the variety to all coordinate lines,
       analogously to the proof in section 3 of \cite{Mau80}. The
       ideal of the projection to one coordinate line is
       principal. Transformation and intersection commute.}

     \short{It is also easy to see that at $p=(p_1,\ldots,p_n)\in L^n$ all
     polynomials in $J$ vanish, where \begin{displaymath}
       p_i=\lim_{\mu\rightarrow \infty}p_{\mu,i}=\sum_{\nu=0}^\infty 
       u_{\nu,i}\cdot t^{-\sum_{j=0}^\nu \omega_{j,i}}\in R_N\subset
       L.
     \end{displaymath}}

   \end{proof}

   \begin{remark}\label{rem:monomialordering}
     The proof is basically an algorithm which allows to compute a point
     $p\in V(J)$ such that $\val(p)=-\omega$. However, if we want to use
     a computer algebra system like \textsc{Singular} for the
     computations, then we have to restrict to generators of $J$ which are
     polynomials in $t^\frac{1}{N}$ as well as in $\ux$. Moreover, we
     should pass from $t^\frac{1}{N}$ to $t$, which can be easily
     done by the $K$-algebra isomorphism
     \begin{displaymath}
       \Psi_N:L[\ux]\longrightarrow L[\ux]:t\mapsto t^N,x_i\mapsto x_i.
     \end{displaymath}
     Whenever we do a transformation which involves rational exponents
     we will clear the denominators using this map with an
     appropriate $N$.

     We will in the course of the algorithm have to compute the
     $t$-initial ideal of $J$ with respect to some $\omega\in\Q^n$,
     and we will do so by a standard basis computation using the
     monomial ordering $>_\omega$, given by
     \begin{align*}
      & t^\alpha\cdot\ux^\beta\;>_\omega\;t^{\alpha'}\cdot\ux^{\beta'} \Longleftrightarrow
     \\&
       -\alpha+\omega\cdot\beta>-\alpha'+\omega\cdot\beta' \mbox{ or } 
            (-\alpha+\omega\cdot\beta=-\alpha'+\omega\cdot\beta'
       \;\mbox{ and }\;
       \ux^\beta\;>\;\ux^{\beta'}),
     \end{align*}
     where $>$ is some fixed global monomial ordering on the monomials
     in $\ux$.
     
   \end{remark}

   \begin{algorithm}[ZDL -- Zero Dimensional Lifting Algorithm]\label{alg:ZDL}
     \textsc{Input:} 
     \begin{minipage}[t]{11cm}       
       $(m,f_1,\ldots,f_k,\omega)\in \N_{>0}\times K[t,\ux]^k\times \Q^n$ such that $\dim(J)=0$  and
       $\omega\in \Trop(J)$ for $J=\langle f_1,\ldots,f_k\rangle_{L[\ux]}$. 
     \end{minipage}
     \\[0.2cm]
     \textsc{Output:} 
     \begin{minipage}[t]{11cm}
       $(N,p)\in\N\times K[t,t^{-1}]^n$ such that $p\big(t^\frac{1}{N}\big)$
       coincides with the first $m$ terms of a solution of $V(J)$ and such
       that $\val(p)=-\omega$.
     \end{minipage}
     \\[0.2cm]
     \textsc{Instructions:}
     \begin{itemize}
     \item Choose $N\geq 1$ such that $N\cdot\omega\in\Z^n$.
     \item FOR $i=1,\ldots,k$ DO $f_i:=\Psi_N(f_i)$.
     \item $\omega:=N\cdot\omega$
     \item IF some $\omega_i>0$ THEN
       \begin{itemize}
       \item FOR $i=1,\ldots,k$ DO $f_i:=\Phi_{\omega}(f_i)\cdot t^{-\ord_t\big(\Phi_{\omega}(f_i)\big)}$.
       \item $\tilde{\omega}:=\omega$.
       \item $\omega:=(0,\ldots,0)$.
       \end{itemize}
     \item Compute a standard basis $(g_1,\ldots,g_l)$ of $\langle
       f_1,\ldots,f_k\rangle_{K[t,\ux]}$  with
       respect to the ordering $>_{\omega}$.
     \item Compute a zero $u\in (K^*)^n$ of $\langle
       \tin_\omega(g_1),\ldots,\tin_\omega(g_l)\rangle_{K[\ux]}$. 
     \item IF $m=1$ THEN $(N,p):=\big(N,u_1\cdot
       t^{-\omega_1},\ldots,u_n\cdot t^{-\omega_n}\big)$. 
     \item ELSE
       \begin{itemize}
       \item Set $G=\big(\gamma_{\omega,u}(f_i)\;\big|\;i=1,\ldots,k\big)$.
       \item FOR $i=1,\ldots,n$ DO
         \begin{itemize}
         \item Compute a generating set $G'$ of  
           $\langle G,x_i\rangle_{K[t,\ux]}:\langle t\rangle^\infty$.
         \item IF $G'\subseteq \langle t,\ux\rangle$ 
           THEN
           \begin{itemize}
           \item $\ux:=\ux\setminus\{x_i\}$
           \item Replace $G$ by a generating set of  $\langle
             G'\rangle\cap K[t,\ux]$.
           \end{itemize}
         \end{itemize}
       \item IF $\ux=\emptyset$ THEN $(N,p):=\big(N,u_1\cdot
         t^{-\omega_1},\ldots,u_n\cdot t^{-\omega_n}\big)$. 
       \item ELSE 
         \begin{itemize}
         \item Compute a point $\omega'$ in the negative orthant of the tropical variety of
           $\langle G\rangle_{L[\ux]}$. 
         \item $(N',p')=ZDL(m-1,G,\omega')$.
         \item $N:=N\cdot N'$.           
         \item FOR $j=1,\ldots,n$ DO
           \begin{itemize}
           \item IF $x_i\in\ux$ THEN
             $p_i:=t^{-\omega_i\cdot N'}\cdot (u_i+p'_i)$.
           \item ELSE $p_i:=t^{-\omega_i\cdot N'}\cdot u_i$.
           \end{itemize}
         \end{itemize}
       \end{itemize}
     \item IF some $\tilde{\omega}_i>0$ THEN
       $p:=\big(t^{-\tilde{\omega}_1}\cdot
       p_1,\ldots,t^{-\tilde{\omega}_n}\cdot p_n\big)$. 
     \end{itemize}     
   \end{algorithm}
   \begin{proof}
     The algorithm which we describe here is basically one recursion
     step in the constructive proof of  Theorem \ref{thm:liftinglemma}
     given above, and thus the correctness follows once we have
     justified why our computations do what is required by the
     recursion step.
Notice that step~$4$ and the last step make an adjusting change of variables to make all $\omega_i$ non-positive in the body of the algorithm. This together with step~$3$ guarantees that $t^{-\omega_i}$ is a polynomial. 

     If we compute a standard basis $(g_1,\ldots,g_l)$
     of $\langle f_1,\ldots,f_k\rangle_{K[t,\ux]}$ with respect to $>_\omega$, then by
     Theorem \ref{thm:stdtin}
     the $t$-initial forms of the $g_i$
     generate the $t$-initial ideal of $J=\langle f_1,\ldots,f_k\rangle_{L[\ux]}$.
     We thus compute a zero $u$ of the $t$-initial ideal as required.

     Next the recursion in the proof of Theorem \ref{thm:liftinglemma}
     requires to find an $\omega\in \big(\Q_{>0}\cup\{\infty\}\big)^n$,
     which is $-\val(q)$ for some $q\in
     V(J)$, and we have to eliminate those components which are zero.
     Note that the solutions with first component zero are the
     solutions of $J+\langle x_1\rangle$. Checking if there is a
     solution with strictly positive valuation amounts by the proof of
     Corollary \ref{cor:tropnegative} to checking if $\big(J+\langle
     x_1\rangle\big)\cap K[[t]][\ux]\subseteq \langle t,\ux\rangle$,
     and the latter is equivalent to $G'\subseteq \langle
     t,\ux\rangle$ by Lemma \ref{lem:sat}.
     If so, we eliminate the variable $x_1$
     from $\langle G'\rangle_{K[t,\ux]}$, which amounts to projecting all
     solutions with first component zero to $L^{n-1}$. We then
     continue with the remaining variables. That way we find a set of
     variables $\{x_{i_1},\ldots,x_{i_s}\}$ such that there is a
     solution of $V(J)$ with strictly positive valuation where
     precisely the other components are zero. 
     
     The rest follows from the constructive proof of Theorem
     \ref{thm:liftinglemma}. 
   \end{proof}

   \begin{lemma}\label{lem:sat}
     Let $f_1,\ldots,f_k\in K[t,\ux]$, $J=\langle
     f_1,\ldots,f_k\rangle_{L[\ux]}$, $I=\langle
     f_1,\ldots,f_k\rangle_{K[t,\ux]}:\langle t\rangle^\infty$, and
     let $G$ be a generating set of $I$.
     Then:
     \begin{displaymath}
       J\cap K[[t]][\ux]\subseteq\langle t,\ux\rangle
       \;\;\;\Longleftrightarrow\;\;\;
       I\subseteq\langle t,\ux\rangle
       \;\;\;\Longleftrightarrow\;\;\;
       G\subseteq\langle t,\ux\rangle.
     \end{displaymath}
   \end{lemma}
   \begin{proof}
     The last equivalence is clear since $I$ is generated by $G$, and
     for the first equivalence it suffices to show that $J\cap
     K[[t]][\ux]=\langle I\rangle_{K[[t]][\ux]}$.

     For this let us consider the following two ideals 
     \bmath
     I'=\langle
     f_1,\ldots,f_k\rangle_{K[[t]][\ux]}:\langle t\rangle^\infty
     \emath
     and
     \bmath
     I''=\langle f_1,\ldots,f_k\rangle_{K[t]_{\langle
         t\rangle}[\ux]}:\langle t\rangle^\infty.
     \emath
     By Lemma \ref{lem:saturation} we know that $J\cap K[[t]][\ux]=I'$
     and by \cite{Mar07} Prop.\ 6.20 we know that $I'=\langle
     I''\rangle_{K[[t]][\ux]}$. It thus suffice to show that
     $I''=\langle I\rangle_{K[t]_{\langle t\rangle}[\ux]}$. Obviously
     $I\subseteq I''$, which proves one inclusion. Conversely, if $f\in
     I''$ then $f$ satisfies a relation of the form
     \begin{displaymath}
       t^m\cdot f\cdot u=\sum_{i=1}^k g_i\cdot f_i,
     \end{displaymath}
     with $m\geq 0$, $u\in K[t]$, $u(0)=1$ and $g_1,\ldots,g_k\in
     K[t,\ux]$. Thus $f\cdot u\in I$ and $f=\frac{f\cdot
       u}{u}\in\langle I\rangle_{K[t]_{\langle t\rangle}[\ux]}$.
   \end{proof}

   \begin{remark}
     In order to compute the point $\omega'$ we may want to compute
     the tropical variety of $\langle G\rangle_{L[\ux]}$. The tropical
     variety can be computed as a subcomplex of a Gr\"obner fan or
     more efficiently by applying Algorithm~5 in \cite{BJSST07} for
     computing tropical bases of tropical curves.
   \end{remark}

   \begin{remark}
     We have implemented the above algorithm in the computer algebra
     system \textsc{Singular} (see \cite{GPS05}) since nearly all of the
     necessary computations are reduced to standard basis computations
     over $K[t,\ux]$ with respect to certain monomial orderings. In
     \textsc{Singular} however we do not have 
     an algebraically closed field $K$ over which we
     can compute the zero $u$ of an ideal. We get around this
     by first computing the absolute minimal associated primes of $\langle
     \tin_\omega(g_1),\ldots,\tin_\omega(g_k)\rangle_{K[t,\ux]}$    
     all of which are maximal by Corollary \ref{cor:dimension:D},
     using the absolute primary decomposition in
     \textsc{Singular}. Choosing one of these maximal ideals we only
     have to adjoin one new variable, say $a$, to realise the field extension
     over which the zero lives, and the minimal polynomial, say $m$, for this
     field extension is provided by the absolute primary
     decomposition. In subsequent steps we might have to enlarge the
     minimal polynomial, but we can always get away with only one new
     variable. 

     The field extension should be
     the coefficient field of our polynomial ring in subsequent
     computations.
     Unfortunately, the program \texttt{gfan} which we
     use in order to compute tropical varieties does not handle field
     extensions. (It would not be a problem to actually implement
     field extensions --- we would not have to come up with new
     algorithms.) But we will see in Lemma \ref{lem:fieldex} that we
     can get away with computing tropical varieties of ideals in the
     polynomial ring over the extension field of $K$ by computing
     just over $K$. 
     More precisely, we want to compute a negative-valued point $\omega'$
     in the tropical variety of a transformed ideal $\gamma_{\omega,u}(J)$.
     Instead, we compute a point
     $(\omega',0)$ in the tropical variety of the ideal
     $\gamma_{\omega,u}(J)+\langle m\rangle$. So to justify this
     it is enough to show that $\omega$ is in the tropical variety of an
     ideal $J\trianglelefteq K[a]/\langle m\rangle\{\{t\}\}[\ux]$ if and only if
     $(\omega,0)$ is in the tropical variety of the ideal
     $J+\langle m\rangle\trianglelefteq K\{\{t\}\}[\ux,a]$. Recall
     that $\omega \in \Trop(J)$ if and 
     only if $\tin_\omega(J)$ contains no monomial, and by Theorem
     \ref{thm:stdtin}, $\tin_\omega(J)$ is equal to $\tin_\omega(J_{R_N})$,
     where $N\in \mathcal{N}(J)$.   
   \end{remark}

   \begin{lemma}\label{lem:fieldex}
     Let $m\in K[a]$ be an irreducible polynomial,
     let $\varphi:K[t^{\frac{1}{N}},\ux,a]\rightarrow
     (K[a]/\langle m\rangle)[t^{\frac{1}{N}},\ux]$ take elements to their classes,
     and let $I\trianglelefteq (K[a]/\langle m\rangle)[t^{\frac{1}{N}},\ux]$. Then
     $\IN_\omega(I)$ contains no monomial if and only if
     $\IN_{(\omega,0)}(\varphi^{-1}(I))$ contains no monomial. In
     particular, the same holds for $\tin_\omega(I)$ and
     $\tin_{(\omega,0)}(\varphi^{-1}(I))$. 
   \end{lemma}
   \begin{proof}
     Suppose $\IN_{(\omega,0)}\varphi^{-1}(I)$ contains a
     monomial. Then there exists an $f\in \varphi^{-1}(I)$ such that
     $\IN_{(\omega,0)}(f)$ is a monomial. The polynomial $\varphi(f)$
     is in $I$. When applying $\varphi$ the monomial
     $\IN_{(\omega,0)}(f)$ maps to a monomial whose coefficient in
     $K[a]/\langle m\rangle$ has a representative $h\in K[a]$ with just one term. The
     representative $h$ cannot be $0$ modulo $\langle m\rangle$ since
     $\langle m\rangle$ does not
     contain a monomial. Thus
     $\varphi\big(\IN_{(\omega,0)(f)}\big)=\IN_{\omega}(\varphi(f))$ is a
     monomial.  

     For the other direction, suppose $\IN_\omega(I)$ contains a
     monomial. We must show that
     $\IN_{(\omega,0)}(\varphi^{-1}(I))$ contains a monomial. This is
     equivalent to showing that
     $(\IN_{(\omega,0)}(\varphi^{-1}(I)):((t^{\frac{1}{N}}\cdot
     x_1\cdots x_n)^\infty)$ contains a  
     monomial.
     By assumption there exists an $f\in I$ such that $\IN_\omega(f)$
     is a monomial. Let $g$ be in $\varphi^{-1}(I)$ such that $g$ maps
     to $f$ under the surjection $\varphi$ and with the further
     condition that the support of $g$ projected to the
     $(t^{\frac{1}{N}},\ux)$-coordinates equals the support of $f$. 
     The initial form $\IN_{(\omega,0)}{(g)}$ is a polynomial with all
     exponent vectors having the same $(t^{\frac{1}{N}},\ux)$ parts as
     $\IN_{\omega}(f)$ does. 
     Let $g'$ be $\IN_{(\omega,0)}(g)$ with the common
     $(t^{\frac{1}{N}},\ux)$-part removed from the monomials, that is
     $g'\in K[a]$. Notice that $\varphi(g')\not=0$. We now have
     $g'\not\in \langle m\rangle$ and hence $\langle g',m\rangle
     =k[a]$ since $\langle m\rangle$ is
     maximal. Now $m$ and $g'$ are contained in
     $(\IN_{(\omega,0)}(\varphi^{-1}(I)):(t^{\frac{1}{N}}\cdot
     x_1\cdots x_n)^\infty)$, 
     implying that
     $(\IN_{(\omega,0)}(\varphi^{-1}(I)):(t^{\frac{1}{N}}\cdot
     x_1\cdots x_n)^\infty)\supseteq K[a]$. 
     This shows that $\IN_{(\omega,0)}(\varphi^{-1}(I))$ contains a monomial.       
   \end{proof}

   \begin{remark}
     In Algorithm \ref{alg:ZDL} we choose zeros of the $t$-initial ideal
     and we choose points in the negative quadrant of the tropical
     variety. If we instead do the same computations for all zeros and
     points of the negative quadrant of the tropical variety, then we
     get Puiseux expansions of all branches of the space curve germ defined
     by the ideal $\langle f_1,\ldots,f_k\rangle_{K[[t,\ux]]}$ in
     $(K^{n+1},0)$.  
   \end{remark}


   \section{Reduction to the Zero Dimensional Case}\label{sec:arbitraryliftinglemma}

   In this section, we want to give a proof of the Lifting Lemma
   (Theorem \ref{thm:liftinglemma}) for any ideal $J$ of dimension $\dim J=d>0$, using
   our algorithm for the zero-dimensional case.
   
   Given $\omega\in\Trop(J)$ we would like to
   intersect $\Trop(J)$ with another tropical variety $\Trop(J')$
   containing $\omega$, such
   that $\dim(J+J')=0$ and apply the zero-dimensional algorithm to
   $J+J'$. However, we cannot conclude that $\omega\in \Trop(J+J')$ --- we
   have $\Trop(J+J')\subseteq \Trop(J)\cap \Trop(J')$ but equality does
   not need to hold. For example, two plane tropical lines (given by
   two linear forms) which are not equal can intersect in a ray, even
   though the ideal generated by the two linear forms defines just a
   point.  
   
%
%

   So we have to find an
   ideal $J'$ such that $J+J'$ is zero-dimensional and still
   $\omega\in\Trop(J+J')$ (see Proposition
   \ref{prop:intersect}). We will use some ideas of \cite{Kat06}
   Lemma 4.4.3 --- the ideal $J'$ will be generated by $\dim(J)$ sufficiently general
   linear forms. The proof of the proposition needs some technical preparations.

   \begin{notation}
     We denote by
     \begin{displaymath}
       V_\omega=\{a_0+a_1\cdot t^{\omega_1}\cdot x_1+\ldots+a_n\cdot t^{\omega_n}\cdot x_n\;|\;a_i\in K\}
     \end{displaymath}
     the $n+1$-dimensional $K$-vector space of \emph{linear} polynomials over
     $K$, which in a sense are \emph{scaled} by $\omega\in\Q^n$. Of
     most interest will be the case where $\omega=0$.
   \end{notation}

   The following lemma geometrically says that an affine variety of
   dimension at least one will intersect a generic hyperplane.

   \begin{lemma}\label{lem:hyperplane}
     Let $K$ be an infinite field and $J\lhd L[\ux]$ an equidimensional
     ideal of dimension $\dim(J)\geq 1$. Then there is a Zariski open dense subset 
     $U$ of $V_0$ such
     that $\langle f\rangle + Q\not=L[\ux]$ for all $f\in U$ and $Q\in\minAss(J)$.
   \end{lemma}

   If $V$ is an affine
   variety which meets $(K^*)^n$ in dimension at least $1$, then a
   generic hyperplane section of $V$ meets $(K^*)^n$ as well. The
   algebraic formulation of this geometric fact is the following
   lemma: 
   
   \begin{lemma}\label{lem:findf}
     Let $K$ be an infinite field and $I\lhd K[\ux]$ be an equidimensional ideal 
     with $\dim(I)\geq 1$ and such that $x_1\cdots x_n\not\in\sqrt{I}$,
     then there is a Zariski open subset 
     $U$ of $V_0$ such
     that $x_1\cdots x_n\not\in \sqrt{I+\langle f\rangle}$ for $f\in U$.
   \end{lemma}

   The following lemma is an algebraic formulation of the geometric fact that given any
   affine variety none of its components will be contained in a generic hyperplane.

   \begin{lemma}\label{lem:findf2}
     Let $K$ be an infinite field, let $R$ be a ring containing $K$, 
     and let $J\unlhd R[\ux]$ be an ideal.
     Then there is
     a Zariski open dense subset  
     $U$ of $V_0$ such
     that $f\in U$ satisfies $f\not\in P$ for $P\in\minAss(J)$.
   \end{lemma}
   \longer{  \begin{proof}
     For $P\in\minAss(J)$ the subspace $W_P=P\cap
     V_0$ of $V_0$ has dimension at most $n$
     since otherwise the prime ideal $P$ would contain $1$. Since $K$
     is infinite the set 
     \begin{displaymath}
       U=V_0\setminus\bigcup_{P\in\minAss(J)}W_Q
     \end{displaymath}
     is a non-empty Zariski open subset of $V_0$, and it
     is thus dense.
   \end{proof}
   }

   \begin{remark}\label{rem:char}
     If $\#K<\infty$ we can still find a suitable $f\in K[\ux]$ which
     satisfies the conditions in Lemma \ref{lem:hyperplane},
     Lemma \ref{lem:findf} and Lemma \ref{lem:findf2}
     due to Prime Avoidance. However, it may not be possible to choose a
     linear one.
   \end{remark}

   With these preparations we can show that we can reduce to the zero
   dimensional case by cutting with generic hyperplanes.

   \begin{proposition}\label{prop:intersect}    
     Suppose that $K$ is an infinite field, and let $J\lhd L[\ux]$ be
     an equidimensional ideal of dimension $d$ and 
     $\omega\in\Trop(J)\cap\Q^n$. 

     Then there exist Zariski open dense subsets $U_1,\ldots,U_d$ of
     $V_\omega$ such that $(f_1,\ldots,f_d)\in
     U_1\times\ldots\times U_d$ and $J'=\langle
     f_1,\ldots,f_d\rangle_{L[\ux]}$ satisfy:
     \begin{itemize}
     \item
       $\dim(J+J')=\dim\big(\tin_\omega(J)+\tin_\omega(J')\big)=0$,
     \item $\dim\big(\tin_\omega(J')\big)=\dim(J')=n-d$,
     \item $x_1\cdots x_n\not\in \sqrt{\tin_\omega(J)+\tin_\omega(J')}$, and
     \item $\sqrt{\tin_\omega(J)+\tin_\omega(J')}=\sqrt{\tin_\omega(J+J')}$.
     \end{itemize}
     In particular, $\omega\in\Trop(J+J')$.
   \end{proposition}
   \begin{proof}
     Applying $\Phi_\omega$ to $J$ first and then applying
     $\Phi_{-\omega}$ to $J'$ later we may assume that
     $\omega=0$. Moreover, we may choose an $N$ such that $N\in
     \mathcal{N}(J)$ and $N\in\mathcal{N}(P)$ for all
     $P\in\minAss(J)$. By Lemma \ref{lem:intin} then also 
     $\tin_0(J)=\tin_0(J_{R_N})$ and $\tin_0(P)=\tin_0(P_{R_N})$ for $P\in
     \minAss(J)$. 

     By Lemma \ref{lem:minAsseqdim}
     \begin{equation}
       \label{eq:intersect:0}
       \minAss(J_{R_N})=\{P_{R_N}\;|\;P\in\minAss(J)\}.
     \end{equation}
      In particular, all minimal associated primes
     $P_{R_N}$ of $J_{R_N}$ have codimension $n-d$ by Corollary \ref{cor:dimension:P}.

     Since $0\in \Trop(J)$ there exists a $P\in\minAss(J)$ with
     $0\in\Trop(P)$ by Lemma \ref{lem:tropicalvariety}. Hence $1\not\in\tin_0(P)$ and we conclude by
     Corollary \ref{cor:minAsstin} 
     that 
     \begin{equation}
       \label{eq:intersect:1}
       \dim(J)=\dim\big(\tin_0(J)\big)
       =\dim(Q)
     \end{equation}
     for all $Q\in\minAss\big(\tin_0(J)\big)$. In particular, all
     minimal associated prime ideals of $\tin_0(J)$ have codimension
     $n-d$. 

     Moreover, since $0\in\Trop(J)$ we know that $\tin_0(J)$
     is monomial free, and in particular
     \begin{equation}\label{eq:intersect:2}
       x_1\cdots x_n\not\in\sqrt{\tin_0(J)}.
     \end{equation}

     If $d=0$ then
     $J'=\langle \emptyset\rangle=\{0\}$ works due to
     \eqref{eq:intersect:1} and \eqref{eq:intersect:2}. We may thus assume that 
     $d>0$.

     Since $K$ is infinite we can apply Lemma \ref{lem:hyperplane} to $J$,
     Lemma \ref{lem:findf2} to $J\lhd L[\ux]$, to $J_{R_N}\lhd R_N[\ux]$ and to
     $\tin_0(J)\lhd K[\ux]$ and Lemma \ref{lem:findf} to
     $\tin_0(J)\lhd K[\ux]$ (take \eqref{eq:intersect:2} into account),
     and thus there exist  Zariski open dense 
     subsets $U$, $U'$, $U''$, $U'''$  and $U''''$ in $V_0$ such
     that no $f_1\in U_1=U\cap U'\cap U''\cap U'''\cap U''''$ is contained in any minimal associated
     prime of either $J$, $J_{R_N}$ or $\tin_0(J)$, such that
     $1\not\in J+\langle f_1\rangle_{L[\ux]}$ and such that
     $x_1\cdots x_n\not\in\sqrt{\tin_0(J)+\langle f_1\rangle}$. Since the intersection
     of four Zariski open and dense subsets is non-empty, there is such an
     $f_1$ and by Lemma \ref{lem:minAss} the minimal associated primes
     of the ideals $J+\langle f_1\rangle_{L[\ux]}$, $J_{R_N}+\langle f_1\rangle_{R_N[\ux]}$, 
     and $\tin_0(J)+\langle f_1\rangle_{K[\ux]}$ all have the same codimension
     $n-d+1$. 


     We claim that $t^\frac{1}{N}\not\in Q$ for any $Q\in\minAss(J_{R_N}+\langle
     f_1\rangle_{R_N[\ux]})$. Suppose the contrary, then
     by Lemma \ref{lem:localisation} (b), (f) and (g)
     \begin{displaymath}
       \dim(Q)=n+1-\codim(Q)=d.
     \end{displaymath}
     Consider now the residue class map
     \begin{displaymath}
       \pi:R_N[\ux]\longrightarrow R_N[\ux]/\big\langle t^\frac{1}{N}\big\rangle=K[\ux].
     \end{displaymath}
     Then $\tin_0(J)=\pi\big(J_{R_N}+\big\langle t^\frac{1}{N}\big\rangle\big)$, and we
     have
     \begin{displaymath}
       \tin_0(J)+\langle f_1\rangle_{K[\ux]}
       \subseteq
       \pi\big(J_{R_N}+\langle t^\frac{1}{N},f_1\rangle_{R_N[\ux]}\big)
       \subseteq
       \pi(Q).
     \end{displaymath}
     Since $t^\frac{1}{N}\in Q$ the latter is again a prime ideal
     of dimension $d$. However, due to the choice of $f_1$ we know
     that every minimal associated prime of $\tin_0(J)+\langle
     f_1\rangle_{K[\ux]}$ has codimension $n-d+1$ and hence the ideal
     itself has dimension $d-1$. But then it cannot be contained in an
     ideal of dimension $d$.

     Applying the same arguments another $d-1$ times we find Zariski
     open dense subsets $U_2,\ldots,U_d$ of $V_0$ such
     that for all $(f_1,\ldots,f_d)\in U_1\times\cdots\times U_d$ 
     the minimal associated primes of the ideals
     \begin{displaymath}
       J+\langle f_1,\ldots,f_k\rangle_{L[\ux]}
     \end{displaymath}
     respectively 
     \begin{displaymath}
       J_{R_N}+\langle f_1,\ldots,f_k\rangle_{R_N[\ux]}
     \end{displaymath}
     respectively
     \begin{displaymath}
       \tin_0(J)+\langle f_1,\ldots,f_k\rangle_{K[\ux]}
     \end{displaymath}
     all have codimension $n-d+k$ for each $k=1,\ldots,d$, such that
     $1\not\in J+\langle f_1,\ldots,f_k\rangle_{L[\ux]}$, and such
     that
     \begin{displaymath}
       x_1\cdots x_n\not\in\sqrt{\tin_0(J)+\langle f_1,\ldots,f_k\rangle_{K[\ux]}}.
     \end{displaymath}
     Moreover, none of the minimal associated primes of $J_{R_N}+\langle
     f_1,\ldots,f_k\rangle_{R_N[\ux]}$ contains $t^\frac{1}{N}$.

     In particular, since $f_i\in K[\ux]$ we have (see
     Theorem \ref{thm:stdtin})
     \begin{displaymath}
       \tin_0(J')=\tin_0\big(\langle
       f_1,\ldots,f_d\rangle_{K[t,\ux]}\big)
       =\langle f_1,\ldots,f_d\rangle_{K[\ux]},
     \end{displaymath}
     and $J'$ obviously satisfies the first three requirements of the
     proposition. 

     For the fourth requirement it suffices to show
     \begin{displaymath}
       \minAss\big(\tin_0(J)+\tin_0(J')\big)=\minAss\big(\tin_0(J+J')\big).
     \end{displaymath}
     For this consider the ring extension
     \begin{displaymath}
       R_N[\ux]\subseteq S_N^{-1}R_N[\ux]=L_N[\ux]
     \end{displaymath}
     given by localisation and denote by $I^c=I\cap R_N[\ux]$ the contraction of an
     ideal  $I$ in $L_N[\ux]$ and by $I^e=\langle
     I\rangle_{L_N[\ux]}$ the extension of an ideal $I$ in
     $R_N[\ux]$. Moreover, we set $J_0=J\cap L_N[\ux]$ and
     $J_0'=J'\cap L_N[\ux]$, so that $J_0^c=J_{R_N}$
     and ${J'_0}^c=\langle f_1,\ldots,f_d\rangle_{R_N[\ux]}$.

     Note then first that
     \begin{displaymath}
       (J_0^c+{J'_0}^{c})^e=
       J_0^{ce}+{J'_0}^{ce}=J_0+J_0',
     \end{displaymath}
     and therefore by the correspondence of primary decomposition
     under localisation (see \cite{AM69} Prop.\ 4.9)
     \begin{displaymath}
       \minAss\big((J_0+J_0')^c\big)=
       \big\{Q\in\minAss(J_0^c+{J'_0}^{c})\;\big|\;t^\frac{1}{N}\not\in Q\big\} 
       =\minAss\big(J_0^c+{J'_0}^{c}\big).
     \end{displaymath}
     This then shows that
     \begin{displaymath}
       \sqrt{J_0^c+{J'_0}^{c}}=\sqrt{(J_0+J_0')^c},
     \end{displaymath}
     and since $\pi(J_0^c)=\tin_0(J_{R_N})=\tin_0(J)$,
     $\pi({J'_0}^c)=\tin_0(J')$
     and $\pi\big((J_0+J_0')^c\big)=\tin_0(J+J')$
     we get \longer{by Lemma \ref{lem:radical}}
     \begin{multline*}
       \sqrt{\tin_0(J)+\tin_0(J')}
       =\sqrt{\pi(J_0^c)+\pi({J'_0}^{c})}
       =\pi\left(\sqrt{J_0^c+{J'_0}^{c}}\right)\\
       =\pi\left(\sqrt{(J_0+J_0')^c}\right)
       =\sqrt{\pi\big((J_0+J_0')^c\big)}
       =\sqrt{\tin_0(J+J')}.
     \end{multline*}

     It remains to show the ``in particular'' part. However, 
     since 
     \begin{displaymath}
       x_1\cdots x_n\not\in \sqrt{\tin_\omega(J)+\tin_\omega(J')}
       =\sqrt{\tin_\omega(J+J')},
     \end{displaymath}
     the ideal $\tin_\omega(J+J')$ is monomial free, or equivalently $\omega\in
     \Trop(J+J')$. 
   \end{proof}

   \longer{
   \begin{lemma}\label{lem:radical}
     Let $\pi:R\longrightarrow R'$ be a surjective ring homomorphism and
     $I\unlhd R$ an ideal. 
     Then $\sqrt{\pi(I)}=\sqrt{\pi(\sqrt{I})}$.
   \end{lemma}
   \begin{proof}
     Since $\pi$ is surjective $\pi$ maps ideals in $R$ to
     ideals in $R'$.  Note that $\pi(I)\subset \pi(\sqrt{I})$, hence
     also $\sqrt{\pi(I)}\subset\sqrt{\pi(\sqrt{I})}$. For the other
     inclusion, let $g\in \pi(\sqrt{I})$. Then there exists an $f\in
     \sqrt{I}$ such that $g=\pi(f)$,
     and there exists a $k$ such that $f^k\in I$. But then
     $\pi(f)^k=\pi(f^k)\in \pi(I)$ and thus $\pi(f)\in
     \sqrt{\pi(I)}$. As the latter ideal is radical, we have
     $\sqrt{\pi(\sqrt{I})}\subset \sqrt{\pi(I)}$. 
   \end{proof}
   } 

   \begin{remark}
     Proposition \ref{prop:intersect} shows that 
     the ideal $J'$ can be found by choosing $d$
     linear forms $f_j=\sum_{i=1}^n a_{ji}\cdot t^{\omega_i}\cdot
     x_i+a_{j0}$ with random $a_{ji}\in K$, and we only need that $K$
     is infinite.
   \end{remark}

   We are now in the position to finish the proof of Theorem
   \ref{thm:2tropdef}. 

   \begin{proof}[Proof of Theorem \ref{thm:2tropdef}]
     If $\omega\in\Trop(J)\cap \Q^n$ then there is a minimal
     associated prime ideal $P\in\minAss(J)$ such that
     $\omega\in\Trop(P)$ by Lemma \ref{lem:tropicalvariety}.
     By assumption the field $K$ is algebraically
     closed and therefore infinite, so that Proposition
     \ref{prop:intersect} applied to $P$ shows that we can choose an
     ideal $P'$ such that $\omega\in\Trop(P+P')$ and
     $\dim(P+P')=0$. By Theorem \ref{thm:liftinglemma} there exists a
     point $p\in V(P+P')\subseteq V(J)$ such that
     $\val(p)=-\omega$. This finishes the proof in view of Proposition
     \ref{prop:tropical}.
   \end{proof}

   \begin{algorithm}[RDZ - Reduction to Dimension Zero]\label{alg:RDZ}
      \textsc{Input:} 
     \begin{minipage}[t]{11cm}       
       a prime ideal $P \in K(t)[\ux]$ and $\omega\in \Trop(P)$. 
     \end{minipage}
     \\[0.2cm]
     \textsc{Output:} 
     \begin{minipage}[t]{11cm}
       an ideal $J$ such that $\dim(J)=0$, $P\subset J$ and $\omega \in \Trop(J)$.
     \end{minipage}
     \\[0.2cm]
     \textsc{Instructions:}     
     \begin{itemize}
     \item $d:=\dim(P)$
     \item $J:=P$
     \item WHILE $\dim(J)\not=0$ OR $\tin_\omega(J)$ not monomial-free
       DO
       \begin{itemize}
       \item FOR $j=0$ TO $d$ pick random values
         $a_{0,j},\ldots,a_{n,j}\in K$, and define $f_j:= a_{0,j}+\sum
         a_{i,j}\cdot t^{\omega_i}x_i$. 
       \item $J:=P+\langle f_1,\ldots,f_d\rangle$
       \end{itemize}
     \end{itemize}
   \end{algorithm}
   \begin{proof}
     We only have to show that the random choices will lead to a suitable ideal $J$ with probability $1$.
     To see this, we want to apply Proposition
     \ref{prop:intersect}. For this we only have to see that
     $P^e=\langle P\rangle_{L[\ux]}$ is equidimensional of dimension
     $d=\dim(P)$.  
     By \cite{Mar07} Corollary 6.13 the intersection of $P^e$ with
     $K(t)[\ux]$, $P^{ec}$, is equal to $P$. 
     Using Proposition \ref{prop:pd} we see that
     \begin{displaymath}
       \{P\}=\minAss(P^{ec})\subseteq\{Q^c\;|\;Q\in\minAss(P^e)\}\subseteq\Ass(P^{ec})=\{P\}.
     \end{displaymath}
     By Lemma
     \ref{lem:dimFF'} we have $\dim Q=\dim(P)=d$ for every $Q\in\minAss(P^e)$, hence
     $P^e$ is equidimensional of dimension $d$. 
   \end{proof}

   \begin{remark}
     Note that we cannot perform primary decomposition over $L[\ux]$
     computationally. Given a $d$-dimensional ideal $J$ and $\omega
     \in \Trop(J)$ in our implementation of the lifting algorithm, we
     perform primary decomposition over $K(t)[\ux]$. By Lemma
     \ref{lem:tropicalvariety}, there must be a minimal associated
     prime $P$ of $J$ such that $\omega \in \Trop(P)$. Its
     restriction to $K(t)[\ux]$ is one of the minimal associated primes that
     we computed, and this prime is our input for algorithm \ref{alg:RDZ}. 
   \end{remark}

   \begin{example}
     Assume $P=\langle x+y+t\rangle \unlhd L[x,y]$, and
     $\omega=(-1,-2)$. Choose coefficients randomly and add for
     example the linear form 
     $f=-2xt^{-1}+2t^{-2}y-1 $. Then $J=\langle x+y+t,f\rangle $ has
     dimension $0$ and $\omega$ is contained in $\Trop(J)$. 
     Note that the intersection of $\Trop(P)$ with $\Trop(f)$ is not
     transversal, as the vertex of the tropical line $\Trop(f)$ is at
     $\omega$. 
   \end{example}


   \section{Some Commutative Algebra}\label{sec:generalcommutativealgebra}

   In this section we gather some simple results from commutative
   algebra for the lack of a better reference. They are primarily
   concerned with the dimension of an ideal under contraction
   respectively extension for certain ring extensions. The results in
   this section are independent of the previous sections


   \begin{notation}
     In this section we denote by $I^e=\langle I\rangle_{R'}$ the
     extension of $I\unlhd R$ and 
     by $J^c=\varphi^{-1}(J)$ the contraction of $J\unlhd R'$, where
     $\varphi:R\rightarrow R'$ is a ring extension. If no ambiguity
     can arise we will not explicitly state the ring extension.
   \end{notation}

   We first want to understand how primary decomposition behaves under
   restriction. The following lemma is an easy consequence of the definitions.

   \begin{lemma}\label{lem:primary}
     If $\varphi:R\rightarrow R'$ is any ring extension and 
     $Q\lhd R'$ a $P$-primary ideal, then $Q^c$ is $P^c$-primary.
   \end{lemma}
   \longer{ \begin{proof}
     If $a,b\in R$ such that $a\cdot b\in Q^c$ then
     $\varphi(a)\cdot \varphi(b)=\varphi(ab)\in Q$. Since $Q$ is
     primary it follows that $\varphi(a)\in Q$ or
     $\varphi(b^n)=\varphi(b)^n\in Q$ for some $n\geq 1$. Hence
     $a\in Q^c$ or $b^n\in Q^c$. Since by assumption $1\not\in Q$
     and thus $1\not\in Q^c$, this implies that $Q^c$ is primary.

     Moreover, if $b\in P^c$ then $\varphi(b)\in P$ and thus
     $\varphi(b^n)=\varphi(b)^n\in Q$ for some $n$. But then $b^n\in
     Q^c$ and therefore $P^c=\sqrt{Q^c}$ since $P^c$ is a prime
     ideal.      
   \end{proof}
   }

   \begin{proposition}\label{prop:pd}
     Let $\varphi:R\rightarrow R'$ be any ring extension, 
     let $J\unlhd R'$ be an ideal such that
     $(J^c)^e=J$, and let $J=Q_1\cap\ldots\cap Q_k$ be a minimal
     primary decomposition. Then
     \begin{displaymath}
       \Ass(J^c)=\big\{P^c\;\big|\;P\in \Ass(J)\big\}
       =\Big\{\sqrt{Q_i}^c\;\Big|\;i=1,\ldots,k\Big\},
     \end{displaymath}
     and 
     \bmath
     J^c=\bigcap_{P\in\Ass(J^c)}Q_P
     \emath
     is a minimal primary decomposition, where
     \begin{displaymath}
       Q_P=\bigcap_{\sqrt{Q_i}^c=P}Q_i^c.
     \end{displaymath}

     Moreover, we have
     \bmath
       \minAss(J^c)\subseteq \big\{P^c\;\big|\;P\in\minAss(J)\big\}.
     \emath

     Note that the $\sqrt{Q_i}^c$ are not necessarily pairwise
     different, and thus the cardinality of $\Ass(J^c)$ may be
     strictly smaller than $k$.
   \end{proposition}
   \begin{proof}
     Let
     $\mathcal{P}=\big\{\sqrt{Q_i}^c\;\big|\;i=1,\ldots,k\big\}$ and
     let $Q_P$ be defined as above for $P\in \mathcal{P}$. 
     Since contraction commutes with intersection we have
     \begin{equation}\label{eq:pd:1}
       J^c=\bigcap_{P\in\mathcal{P}}Q_P.
     \end{equation}
     By Lemma \ref{lem:primary}
     the $Q_i^c$ with $P=\sqrt{Q_i}^c$ are $P$-primary, and thus so is
     their intersection, so that  \eqref{eq:pd:1} is a primary
     decomposition. Moreover, by construction the radicals of the
     $Q_P$ are pairwise different. It thus remains to show that none
     of the $Q_P$ is superfluous. Suppose that there is a
     $P=\sqrt{Q_i}^c\in \mathcal{P}$ such that 
     \begin{displaymath}
       J^c=\bigcap_{P'\in\mathcal{P}\setminus\{P\}}Q_{P'}\subseteq\bigcap_{j\not=i}Q_j^c,
     \end{displaymath}
     then
     \begin{displaymath}
       J=(J^c)^e\subseteq \bigcap_{j\not=i}(Q_j^c)^e\subseteq
       \bigcap_{j\not=i}Q_j
     \end{displaymath}
     in contradiction to the minimality of the given primary
     decomposition of $J$. This shows that \eqref{eq:pd:1} is a
     minimal primary decomposition and that $\Ass(J^c)=\mathcal{P}$.

     Finally, if $P\in\Ass(J)$ such that $P^c$ is minimal over $J^c$
     then necessarily there is a $\tilde{P}\in\minAss(J)$ such that
     $P^c=\tilde{P}^c$. 

   \end{proof}

   We will use this result to show that dimension behaves well under
   extension for polynomial rings over a field extension.

   \begin{lemma}\label{lem:dimFF'}
     If $F\subseteq F'$ is a field extension, $I\unlhd F[\ux]$ is an ideal and
     $I^e=\langle I\rangle_{F'[\ux]}$ then 
     \begin{displaymath}
       \dim(I^e)=\dim(I).
     \end{displaymath}
     Moreover, if $I$ is prime then $\dim(P)=\dim(I)$ for all
     $P\in\minAss(I^e)$. 
   \end{lemma}
   \begin{proof}
     Choose any global degree ordering $>$ on the monomials in $\ux$
     and compute a standard basis $G'$ of $I$ with respect to $>$. Then
     $G'$ is also a standard basis of $I^e$ by Buchberger's Criterion. If
     $M$ is the set of leading monomials of elements of $G'$ with
     respect to $>$, then the dimension of the ideal generated by $M$
     does not depend on the base field but only on $M$ (see e.g.\
     \cite{GP02} Prop.\ 3.5.8). Thus we have (see e.g.\ \cite{GP02}
     Cor.\ 5.3.14)
     \begin{equation}\label{eq:dimFF':1}
       \dim(I)=\dim\big(\langle M\rangle_{F[\ux]}\big)
       =\dim\big(\langle M\rangle_{F'[\ux]}\big)=\dim(I^e).
     \end{equation}

     Let now $I$ be prime. It remains to
     show that $I^e$ is equidimensional.

     If we choose a maximal independent set $\ux'\subseteq\ux$ of
     $L_>(I^e)=\langle M\rangle_{F'[\ux]}$ then by definition 
     (see \cite{GP02} Def.\ 3.5.3)  $\langle
     M\rangle\cap F'[\ux']=\{0\}$, so that necessarily $\langle
     M\rangle_{F[\ux]}\cap F[\ux']=\{0\}$. This shows that $\ux'$ is
     an independent set of $L_>(I)=\langle M\rangle_{F[\ux]}$, and it
     is maximal since its size is $\dim(I^e)=\dim(I)$ by
     \eqref{eq:dimFF':1}. Moreover, by \cite{GP02} Ex.\ 3.5.1 $\ux'$
     is a maximal independent set of both $I$ and $I^e$. Choose now a
     global monomial ordering $>'$ on the monomials in
     $\ux''=\ux\setminus\ux'$. 

     We claim that if $G=\{g_1,\ldots,g_k\}\subset F[\ux]$ is a standard basis of
     $\langle I\rangle_{F(\ux')[\ux'']}$ with respect to $>'$ and if 
     \bmath
     0\not=h=\lcm\big(\lc_{>'}(g_1),\ldots,\lc_{>'}(g_k)\big)\in F[\ux'],
     \emath
     then $I^e:\langle h\rangle^\infty=I^e$. For this we consider a
     minimal primary decomposition $I^e=Q_1\cap\ldots\cap Q_l$ of
     $I^e$. Since $I^{ece}=I^e$ we may apply Proposition \ref{prop:pd}
     to get
     \begin{equation}\label{eq:dimFF':2}
       \big\{\sqrt{Q_i}^c\;\big|\;i=1,\ldots,l\big\}=\Ass(I^{ec})=\{I\},
     \end{equation}
     where the latter equality is due to $I^{ec}=I$ (see
     e.g. \cite{Mar07} Cor.\ 6.13) and to $I$ being prime. Since
     $\ux'$ is an independent set of $I$ we know that $h\not\in I$ and
     thus \eqref{eq:dimFF':2} shows that $h^m\not\in \sqrt{Q_i}$ for any
     $i=1,\ldots,l$ and any $m\in\N$. Let now $f\in I^e:\langle
     h\rangle^\infty$, then there is an $m\in\N$ such that $h^m\cdot
     f\in I^e\subseteq Q_i$ and since $Q_i$ is primary and $h^m\not\in
     \sqrt{Q_i}$ this forces $f\in Q_i$. But then $f\in
     Q_1\cap\ldots\cap Q_l=I^e$, which proves the claim.

     With the same argument as at the beginning of the proof we see
     that $G$ is a standard basis of $\langle
     I^e\rangle_{F'(\ux')[\ux'']}$, and we may thus apply \cite{GP02}
     Prop.\ 4.3.1 to the ideal $I^e$ which shows that $I^e:\langle
     h\rangle^\infty$ is equidimensional. We are thus done by the
     claim. 
   \end{proof}

   If the field extension is algebraic then dimension also behaves
   well under restriction.

   \begin{lemma}\label{lem:dimLLN}
     Let $F\subseteq F'$ be an algebraic field extension and let
     $J\lhd F'[\ux]$ be an ideal, then
     \bmath
     \dim(J)=\dim(J\cap F[\ux]).
     \emath
   \end{lemma}
   \begin{proof}
     Since the field extension is algebraic the ring extension
     $F[\ux]\subseteq F'[\ux]$ is integral again. But then the ring extension 
     \bmath
     F[\ux]/J\cap F[\ux]\hookrightarrow F'[\ux]/J
     \emath
     is integral again (see \cite{AM69} Prop.\ 5.6), and in particular
     they have the same dimension  (see \cite{Eis96} Prop.\ 9.2).
   \end{proof}

   For Section~\ref{sec:arbitraryliftinglemma} --- where we want to
   intersect an ideal of arbitrary dimension to get a zero-dimensional
   ideal --- we need to understand how dimension behaves when we
   intersect. The following \short{result is}\longer{results are} 
   concerned with that question. Geometrically \short{it just means}\longer{they just mean} that intersecting an equidimensional
   variety with a hypersurface which does not contain any irreducible
   component leads again to an equidimensional variety of dimension
   one less. We need this result over $R_N$ instead of a
   field $K$.

   \begin{lemma}\label{lem:minAss}
     Let $R$ be a catenary integral domain, let $I\lhd R$ with $\codim(Q)=d$ for
     all $Q\in\minAss(I)$, and let
     $f\in R$ such that  $f\not\in Q$ for all $Q\in\minAss(I)$.
     Then 
     \begin{displaymath}
       \minAss(I+\langle f\rangle)=\bigcup_{Q\in
         \minAss(I)}\minAss(Q+\langle f\rangle).
     \end{displaymath}
     In particular, $\codim(Q')=d+1$ for all $Q'\in\minAss(I+\langle f\rangle)$.
   \end{lemma}
   \begin{proof}
     If $Q'\in\minAss(I+\langle f\rangle)$ then $Q'$ is
     minimal among the prime ideals containing $I+\langle
     f\rangle$. Moreover, since $I\subseteq Q'$ there is a minimal
     associated prime $Q\in\minAss(I)$ of $I$ which is contained in
     $Q'$. And, since $f\in Q'$ we have $Q+\langle
     f\rangle\subseteq Q'$ and $Q'$ must be minimal with this
     property since it is minimal over $I+\langle f\rangle$. Hence
     $Q'\in \minAss(Q+\langle f\rangle)$.
     
     Conversely, if $Q'\in\minAss(Q+\langle f\rangle)$ where
     $Q\in\minAss(I)$, then $I+\langle f\rangle \subseteq
     Q'$. Thus there exists a $Q''\in\minAss(I+\langle f\rangle)$
     such that $Q''\subseteq Q'$. Then $I\subseteq Q''$ and
     therefore there exists a $\tilde{Q}\in\minAss(I)$ such that
     $\tilde{Q}\subseteq Q''$. Moreover, since $f\not\in \tilde{Q}$
     but $f\in Q''$ this inclusion is strict which implies
     \begin{displaymath}
       \codim(Q')\geq\codim(Q'')\geq \codim\big(\tilde{Q}\big)+1=\codim(Q)+1,
     \end{displaymath}
     where the first inequality comes from $Q''\subseteq Q'$ and 
     the last equality is due to our assumption on $I$.
     But by Krull's Principal Ideal Theorem (see \cite{AM69} Cor.\
     11.17) we have
     \begin{displaymath}
       \codim(Q'/Q)=1,
     \end{displaymath}
     since $Q'/Q$ by assumption is minimal over $f$ in $R/Q$
     where $f$ is neither a unit (otherwise $Q+\langle f\rangle=R$ and
     no $Q'$ exists) nor a zero divisor. Finally, since
     $R$ is catenary and thus all maximal chains of prime
     ideals from $\langle 0\rangle$ to $Q'$ have the same length
     \longer{(here we use that $R$ is an integral domain)} this implies
     \begin{equation}
       \label{eq:minAss:1}
       \codim(Q')=\codim(Q)+1.
     \end{equation}
     This forces that $\codim(Q')=\codim(Q'')$ and thus
     $Q'=Q''\in\minAss(I+\langle f\rangle)$.
   
     The ``in particular'' part follows from \eqref{eq:minAss:1}.
   \end{proof}

   \longer{
   An immediate consequence is the following corollary.

   \begin{corollary}\label{cor:minAss}
     Let $F$ be a field, $I\lhd F[\ux]$ an equidimensional ideal and
     $f\in F[\ux]\setminus F^*$ such that  $f\not\in Q$ for $Q\in\minAss(I)$.
     Then 
     \begin{displaymath}
       \minAss(I+\langle f\rangle)=\bigcup_{Q\in
         \minAss(I)}\minAss(Q+\langle f\rangle).
     \end{displaymath}
     In particular,
     $I+\langle f\rangle$ is equidimensional of dimension
     $\dim(I)-1$.     
   \end{corollary}
   }


   \section{Good Behaviour of the Dimension}\label{sec:dimension}

   In this section we want to show (see Theorem \ref{thm:dimension:C})
   that for an ideal $J\unlhd L[\ux]$, $N\in\mathcal{N}(J)$ and
   a point 
   \bmath
     \omega\in \Trop(P)\cap\Q_{\leq 0}^n
   \emath
   in the non-positive quadrant of the tropical variety of an
   associated prime $P$ of maximal dimension  we have
   \begin{displaymath}
     \dim(J_{R_N})=\dim\big(\tin_\omega(J)\big)+1=\dim(J)+1.
   \end{displaymath}
   The results in this section are independent of Sections
   \ref{sec:basicnotation}, \ref{sec:zerodimensionalliftinglemma} and
   \ref{sec:arbitraryliftinglemma}. 
  
   Let us first give  examples which show that the hypotheses on
   $\omega$ are necessary.

   \begin{example}\label{ex:badT}
     Let $J=\langle 1+tx\rangle\lhd L[x]$ and consider 
     $\omega=1\in\Trop(J)$. Then $\tin_\omega(J)=\langle
     1+x\rangle$ has dimension zero in $K[x]$, and 
     \begin{displaymath}
       I=J\cap R_1[x]=\langle 1+tx\rangle_{R_1[x]}
     \end{displaymath}
     has dimension zero as well by Lemma \ref{lem:localisation} (d).
   \end{example}

   \begin{example}
     Let $J=\langle x-1\rangle \lhd L[x]$ and
     $\omega=-1\not\in\Trop(J)$, 
     then $\tin_\omega(J)=\langle 1\rangle$ has dimension $-1$, while
     $J\cap R_1[x]=\langle x-1\rangle$ has dimension $1$.
   \end{example}

   \begin{example}
     Let $J=P\cdot Q=P\cap Q\lhd L[x,y,z]$ with $P=\langle tx-1\rangle$
     and $Q=\langle x-1,y-1,z-1\rangle$, and let $\omega=(0,0,0)\in
     \Trop(Q)\cap\Q_{\leq 0}^3$. Then $\tin_\omega(J)=\langle
     x-1,y-1,z-1\rangle\lhd K[x,y,z]$ has dimension zero, while
     \begin{displaymath}
       J\cap R_1[x,y,z]=(P\cap R_1[x,y,z])\cap (Q\cap R_1[x,y,z])
     \end{displaymath}
     has dimension two by Lemma \ref{lem:localisation} (d).
   \end{example}

   \longer{
   \begin{remark}
     We will see in Lemma \ref{lem:localisation} that for a prime
     ideal $P\unlhd L[\ux]$  
     \begin{displaymath}
       \dim(P)=\dim(P_{R_N})+1
       \;\;\;\Longleftrightarrow\;\;\;
       1\not\in\IN_0(P_{R_N})
     \end{displaymath}
     while otherwise the dimension stays constant. We have already
     encountered the latter behaviour in Example \ref{ex:badT}, and the main
     reason why things work out fine when $\omega$ lies in the
     negative orthant of $\Trop(P)$ is that then 
     $1\not\in\IN_0(P_{R_N})$, as we will show in the next
     lemma.
   \end{remark}
   }
   
   Before now starting with studying the behaviour of dimension we
   have to collect some technical results used throughout the proofs. 

   \begin{lemma}\label{lem:T}
     Let $J\unlhd L[\ux]$ be an ideal and 
     $\Trop(J)\cap\Q_{\leq 0}^n\not=\emptyset$, then
     $1\not\in\IN_0(J_{R_N})$. 
   \end{lemma}
   \begin{proof}
     Let $\omega\in \Trop(J)\cap\Q_{\leq 0}^n$ and 
     suppose that $f\in J_{R_N}$ with $\IN_0(f)=1$. If
     $t^\alpha\cdot\ux^\beta$ is a monomial of
     $f$ with $t^\alpha\cdot\ux^\beta\not=1$, then $\IN_0(f)=1$
     implies $\alpha>0$,  
     and hence
     $-\alpha+\beta_1\cdot \omega_1+\ldots+\beta_n\cdot \omega_n<0$,  
     since $\omega_1,\ldots,\omega_n\leq 0$ and
     $\beta_1,\ldots,\beta_n\geq 0$. But this shows that
     $\IN_\omega(f)=1$, and therefore $1\in\tin_\omega(J)$, in
     contradiction to our assumption that $\tin_\omega(J)$ is monomial
     free. 
   \end{proof}

   \begin{lemma}\label{lem:saturated}
     Let $I\unlhd R_N[\ux]$ be an ideal such that $I=I:\big\langle
     t^\frac{1}{N}\big\rangle^\infty$ and let $P\in\Ass(I)$, then $P=P:\big\langle
     t^\frac{1}{N}\big\rangle^\infty$ and $t^\frac{1}{N}\not\in P$.
   \end{lemma}
   \begin{proof}
     Since $R_N[\ux]$ is noetherian and $P$ is an associated prime
     there is an $f\in R_N[\ux]$ such that $P=I:\langle f\rangle$ 
     (see \cite{AM69} Prop.\ 7.17). 

     Suppose that $t^\frac{\alpha}{N}\cdot g\in P$ for some $g\in
     R_N[\ux]$ and $\alpha>0$. Then $t^\frac{\alpha}{N}\cdot g\cdot f\in I$, and
     since $I$ is saturated with respect to $t^\frac{1}{N}$ it follows
     that $g\cdot f\in I$. This, however, implies that $g\in P$. Thus
     $P$ is saturated with respect to $t^\frac{1}{N}$. If
     $t^\frac{1}{N}\in P$ then $1\in P$, which contradicts 
     the fact that $P$ is a prime ideal.
   \end{proof}

   Contractions of ideals in $L[\ux]$ to $R_N[\ux]$ are always
   $t^\frac{1}{N}$-saturated. 

   \begin{lemma}\label{lem:saturation}
     Let $I\unlhd R_N[\ux]$ be an ideal in $R_N[\ux]$ and $J=\langle I\rangle_{L[\ux]}$,
     then
     \bmath
       J_{R_N}=I:\big\langle t^\frac{1}{N}\big\rangle^\infty.
     \emath
   \end{lemma}
   \begin{proof}
     Since $L_N\subset L$ is a field extension \cite{Mar07} Corollary 6.13
     implies     
     \bmath
       J\cap L_N[\ux]=\langle I\rangle_{L_N[\ux]},
     \emath
     and it suffices to see that 
     \bmath
       \langle I\rangle_{L_N[\ux]}\cap R_N[\ux]=I:\big\langle t^\frac{1}{N}\big\rangle^\infty.
     \emath
     If $I\cap S_N\not=\emptyset$ then both sides of the equation
     coincide with $R_N[\ux]$, so that we may assume that $I\cap S_N$ is
     empty. Recall that $L_N=S_N^{-1}R_N$, so that if $f\in R_N[\ux]$
     with $t^\frac{\alpha}{N}\cdot f\in I$ for some $\alpha$, then
     \begin{displaymath}
       f=\frac{t^\frac{\alpha}{N}\cdot f}{t^\frac{\alpha}{N}}\in \langle I\rangle_{L_N[\ux]}\cap R_N[\ux].
     \end{displaymath}
     Conversely, if 
     \begin{displaymath}
       f=\frac{g}{t^\frac{\alpha}{N}}\in \langle I\rangle_{L_N[\ux]}\cap R_N[\ux]
     \end{displaymath}
     with $g\in I$, then $g=t^\frac{\alpha}{N}\cdot f\in I$ and thus $f$ is
     in the right hand side.               
   \end{proof}

   \begin{lemma}\label{lem:intin}
     Let $J\unlhd L[\ux]$ and $N\in \mathcal{N}(J)$. Then
     \bmath
     \tin_0(J)=\tin_0(J_{R_N}),
     \emath
     and
     \begin{displaymath}
       1\not\in\tin_0(J)
       \;\;\;\Longleftrightarrow\;\;\;
       1\not\in\IN_0(J_{R_N}).
     \end{displaymath}   
   \end{lemma}
   \begin{proof}
        Suppose that $f\in J_{R_N}\subset J$ then $\tin_0(f)\in
     \tin_0(J)$, and if in addition $\IN_0(f)=1$, then by definition
     $1=\tin_0(f)\in\tin_0(J)$. 

     Let now $f\in J$, then by assumption there are $f_1,\ldots,f_k\in R_{N\cdot
       M}[\ux]$ for some $M\geq 1$, $g_1,\ldots,g_k\in J_{R_N}$ and some $\alpha\geq 0$ such that
     \begin{displaymath}
       t^\frac{\alpha}{M\cdot N}\cdot f=f_1\cdot g_1+\ldots+f_k\cdot
       g_k\in R_{N\cdot M}[\ux].
     \end{displaymath}
     By \cite{Mar07} Corollary 6.17 
     we thus get 
     \begin{displaymath}
       \tin_0(f)=\tin_0\big(t^\frac{\alpha}{N\cdot M}\cdot f\big)
       \in\tin_0( J_{R_{N\cdot M}})
       =\tin_0( J_{R_N}).
     \end{displaymath}
     Moreover, if we assume that
     $1=\tin_0(f)=\tin_0\big(t^\frac{\alpha}{N\cdot M}\cdot f\big)$
     then there is an $\alpha'\geq 0$ such that
     \begin{displaymath}
       t^\frac{\alpha'}{M\cdot N}\cdot \tin_0(f)=
       \IN_0\big(t^\frac{\alpha}{N\cdot M}\cdot f\big)
       \in\IN_0(J_{R_{N\cdot M}}).
     \end{displaymath}
     This necessarily implies that each monomial in
     $t^\frac{\alpha}{N\cdot M}\cdot f$ is divisible by
     $t^\frac{\alpha'}{N\cdot M}$, or by Lemma \ref{lem:saturated} equivalently that
     \bmath
       t^\frac{\alpha-\alpha'}{N\cdot M}\cdot f\in J_{R_{N\cdot M}}.
     \emath
     But then
     \begin{displaymath}
       1=\IN_0\big(t^\frac{\alpha-\alpha'}{N\cdot M}\cdot f\big)\in
       \IN_0(J_{R_{N\cdot M}}),
     \end{displaymath}
     and thus by \cite{Mar07} Corollary 6.19 
     also
     \bmath
       1\in \IN_0(J_{R_N}).
     \emath
   \end{proof}

   In the following lemma we gather the basic information on the ring
   $R_N[\ux]$ which is necessary to understand how the dimension of an ideal
   in $L[\ux]$ behaves when restricting to $R_N[\ux]$.

   \begin{lemma}\label{lem:localisation}
     Consider the ring extension $R_N[\ux]\subset L_N[\ux]$. Then: 
     \begin{enumerate}
     \item $R_N$ is universally catenary, and thus $R_N[\ux]$ is
       catenary. \label{a}
     \item \label{b} If $I\unlhd R_N[\ux]$, then the following are equivalent:
       \begin{enumerate}
       \item $1\not\in\IN_0(I)$.
       \item $\forall\;p\in R_N[\ux]\;:\;1+t^{\frac{1}{N}}\cdot p\not\in I$.
       \item $I+\big\langle t^\frac{1}{N}\big\rangle\subsetneqq R_N[\ux]$. 
       \item $\exists\;P\lhd R_N[\ux]$ maximal such that $I\subseteq
         P$ and $t^\frac{1}{N}\in P$.
       \item $\exists\;P\lhd R_N[\ux]$ maximal such that $I\subseteq
         P$ and $1\not\in\IN_0(P)$.
       \end{enumerate}
       In particular, if $P\lhd R_N[\ux]$ is a maximal ideal, then
       \begin{displaymath}
         1\not\in\IN_0(P)\;\;\;\Longleftrightarrow\;\;\;t^\frac{1}{N}\in P.
       \end{displaymath}
     \item \label{c} If $P\lhd R_N[\ux]$ is a maximal ideal such that $1\not\in\IN_0(P)$,
       then every maximal chain of prime ideals
       contained in $P$ has length $n+2$.
     \item \label{d} If $I\unlhd R_N[\ux]$ is any ideal with $1\in\IN_0(I)$, then
       $R_N[\ux]/I\cong L_N[\ux]/\langle I\rangle$, and $I\cap S_N=\emptyset$ unless
       $I=R_N[\ux]$. In particular, $\dim(I)=\dim\big(\langle
       I\rangle_{L_N[\ux]}\big)$. 
     \item \label{e} If $P\lhd R_N[\ux]$ is a maximal ideal such that $1\in\IN_0(P)$,
       then every maximal chain of prime ideals
       contained in $P$ has length $n+1$.
     \item \label{f} $\dim(R_N[\ux])=n+1$.
     \item \label{g} If $P\lhd R_N[\ux]$ is a prime ideal such that $1\not\in\IN_0(P)$, then  
       \begin{displaymath}
         \dim(P)+\codim(P)=\dim(R_N[\ux])=n+1.
       \end{displaymath}         
     \item \label{h} If $P\lhd R_N[\ux]$ is a prime ideal such that $1\in\IN_0(P)$, then  
       \begin{displaymath}
         \dim(P)+\codim(P)=n.
       \end{displaymath}         
     \end{enumerate}
   \end{lemma}
   \begin{proof}
     For (a), see \cite{Mat86} Thm.\ 29.4.
     
     In (b), the equivalence of (1) and (2) is obvious from the
     definitions. Let us now use this to show that for a maximal
     ideal $P\lhd R_N[\ux]$
     \begin{displaymath}
       1\not\in\IN_0(P)\;\;\;\Longleftrightarrow\;\;\;t^\frac{1}{N}\in P.
     \end{displaymath}

     If $t^\frac{1}{N}\not\in P$ then $t^\frac{1}{N}$ is a unit
     in the field $R_N[\ux]/P$ 
     and thus there is a $p\in R_N[\ux]$ such that $1\equiv t^\frac{1}{N}\cdot p
     \pmod{P}$, or equivalently that $1-t^\frac{1}{N}\cdot p \in P$.
     If on the other hand $t^\frac{1}{N}\in P$ then $1+t^\frac{1}{N}\cdot p\in P$
     would imply that $1=(1+t^\frac{1}{N}\cdot p)-t^\frac{1}{N}\cdot
     p\in
     P$. 

     This proves the claim and shows at the same time the
     equivalence of (4) and (5). 

     If there is a maximal ideal $P$ containing $I$ and such
     that $1\not\in\IN_0(P)$, then of course also
     $1\not\in\IN_0(I)$. Therefore (5) implies (1).

     Let now $I$ be an  ideal such that $1\not\in\IN_0(I)$. 
     Suppose that $I+\langle t^\frac{1}{N}\rangle=R_N[\ux]$. Then
     $1=q+t^\frac{1}{N}\cdot p$ with $q\in I$ and $p\in R_N[\ux]$, and thus 
     $q=1-t^\frac{1}{N}\cdot p\in I$,
     which contradicts our assumption. Thus $I+\langle
     t^\frac{1}{N}\rangle\not= R_N[\ux]$, and (1) implies (3).

     Finally, if $I+\langle
     t^\frac{1}{N}\rangle\not= R_N[\ux]$, then there exists a maximal ideal $P$
     such that $I+\langle t^\frac{1}{N}\rangle\subseteq P$. This
     shows that (3) implies (4), and we are done. 
     
     To see (c), note that if $1\not\in\IN_0(P)$, then $t^\frac{1}{N}\in P$ by (b), and we may
     consider the surjection
     $\psi:R_N[\ux]\longrightarrow R_N[\ux]/\langle t^\frac{1}{N}\rangle=K[\ux]$.
     The prime ideals of $K[\ux]$ are in $1:1$-correspondence
     with those prime ideals of $R_N[\ux]$ which contain $t^\frac{1}{N}$. In
     particular, 
     $P/\langle t^\frac{1}{N}\rangle=\psi(P)$
     is a maximal ideal of $K[\ux]$ and thus any maximal chain of
     prime ideals in $P$ which starts with $\langle t^\frac{1}{N}\rangle$, say
     $\langle t^\frac{1}{N}\rangle =P_0\subset\ldots \subset P_n=P$
     has
     precisely $n+1$ terms since every maximal chain of prime ideals
     in $K[\ux]$ has that many terms. Moreover, by
     Krull's Principal Ideal Theorem (see e.g.\ \cite{AM69}
     Cor.\ 11.17) the prime ideal $\langle
     t^\frac{1}{N}\rangle$ has codimension $1$, so that the chain of prime
     ideals
     \begin{displaymath}
       \langle 0\rangle\subset\langle t^\frac{1}{N}\rangle =P_0\subset\ldots \subset P_n=P
     \end{displaymath}
     is maximal. Since by (a) the ring $R_N[\ux]$ is catenary every
     maximal chain of prime ideals in between $\langle 0\rangle$ and
     $P$ has the same length $n+2$. 
     
     For (d),  we assume that there exists an element $1+t^\frac{1}{N}\cdot
     p\in I$ due to (b). But then 
     $t^\frac{1}{N}\cdot (-p)\equiv 1\pmod{I}$.
     Thus the elements of $S_N=\big\{1,t^\frac{1}{N},t^\frac{2}{N},\ldots\big\}$ are invertible
     modulo $I$. Therefore
     \begin{displaymath}
       \;\;\;\;\;R_N[\ux]/I\cong S_N^{-1} (R_N[\ux]/I)\cong
       S_N^{-1}R_N[\ux]/S_N^{-1}I=L_N[\ux]/\langle I\rangle.
     \end{displaymath}
     In particular, if $I\not=R_N[\ux]$ then $\langle I\rangle\not=L_N[\ux]$ and thus
     $I\cap S_N=\emptyset$.
     
     To show (e), note that by assumption there is an element $1+t^\frac{1}{N}\cdot
     p\in P$ due to (b), and
     since $P$ is maximal $p\not\in R_N$. Choose a prime ideal $Q$
     contained in $P$ which is minimal w.r.t.\ the property that it
     contains $1+t^\frac{1}{N}\cdot p$. Since $1+t^\frac{1}{N}\cdot p$ is neither a unit nor a zero
     divisor Krull's Principal Ideal Theorem \longer{(see e.g.\ \cite{AM69}
       Cor.\ 11.17)} implies that $\codim(Q)=1$. Moreover, since $Q\cap
     S_N=\emptyset$ by Part (d) the ideal $\langle Q\rangle_{L_N[\ux]}$ is a prime ideal which
     is minimal over $1+t^\frac{1}{N}\cdot p$ by the one-to-one correspondence of
     prime ideals under localisation. Since every maximal
     chain of primes in $L_N[\ux]$ has length $n$ \longer{(see e.g.\
     \cite{Eis96} Chap.\ 13, Thm. A)},
     and by Part (d) we have
     $\dim(Q)=\dim\big(\langle Q\rangle_{L_N[\ux]}\big)=n-1$.
     Hence there is a maximal chain of prime ideals of length $n$
     from $\langle Q\rangle_{L_N[\ux]}$ to $\langle P\rangle_{L_N[\ux]}$.
     Since $\codim(Q)=1$ it follows that there is a chain of
     prime ideals of length $n+1$ starting at $\langle 0\rangle$ and
     ending at $P$ which cannot be prolonged. But by (a) the ring
     $R_N[\ux]$ is catenary, and thus every maximal chain of prime
     ideals in $P$ has length $n+1$. 
     
     Claim (f) follows from (c) and (e)\longer{; alternatively see
     \cite{AM69} Ex.\ 11.7}. 
     
     To see (g), note that by (b) there exists a maximal ideal $Q$ containing
     $P$ and $t^\frac{1}{N}$. If
     $k=\codim(P)$ then we may choose a maximal chain of prime
     ideals of length $k+1$ in $P$, and we may prolong it by at most
     $\dim(P)$ prime ideal to a
     maximal chain of prime ideals in $Q$, which by (b) and (c) has length
     $n+2$. Taking (f) into account this shows that 
     \begin{displaymath}
       \dim(P)\geq (n+2)-(k+1)=\dim(R_N[\ux])-\codim(P).
     \end{displaymath}
     However, the converse inequality always holds, which finishes
     the proof.
     
     For (h) note that by (b) there is no maximal ideal which contains
     $t^\frac{1}{N}$ so that every maximal ideal containing $P$ has
     codimension $n$. The result then follows as in (g).
   \end{proof}

   \begin{corollary}\label{cor:dimension:P}
     Let $P\lhd L[\ux]$ be a prime ideal and $N\geq 1$, then 
     \begin{align*}
       \dim(P_{R_N})=\dim(P)+1
       &\;\;\;\Longleftrightarrow\;\;\;
       1\not\in\IN_0(P_{R_N}),\;\;\mbox{ and} \\
           \dim(P_{R_N})=\dim(P)
       &\;\;\;\Longleftrightarrow\;\;\;
       1\in\IN_0(P_{R_N}).
     \end{align*}
     In any case 
     \begin{displaymath}
       \codim(P_{R_N})=\codim(P).
     \end{displaymath}
   \end{corollary}
   \begin{proof}
     Since the field extension $L_N\subset L$ is algebraic by Lemma
     \ref{lem:dimLLN} we have 
     \begin{equation}\label{eq:dim:1}
       \dim(P)=\dim\big(P\cap L_N[\ux]\big)
     \end{equation}
     in any case. If $1\in\IN_0\big(P_{R_N}\big)$, then  Lemma
     \ref{lem:localisation}(d) implies 
     \begin{displaymath}
       \dim\big(P_{R_N}\big)
       =\dim\big(\langle P_{R_N}\rangle_{L_N[\ux]}\big)
       =\dim\big(P\cap L_N[\ux]\big),
     \end{displaymath}
     since $L_N[\ux]$ is a localisation of $R_N[\ux]$.
 
     It thus suffices to show
     that $\dim\big(P_{R_N}\big)=\dim(P)+1$ if
     $1\not\in\IN_0\big(P_{R_N}\big)$.

     Since $P\not=L[\ux]$ we know that $S_N\cap P=\emptyset$. The
     $1:1$-correspondence of prime ideals under localisation thus
     shows that
     \begin{displaymath}
       l:=\codim\big(P\cap L_N[\ux]\big)=\codim\big(P_{R_N}\big).
     \end{displaymath}
     Hence there exists a maximal chain of prime ideals
     \begin{displaymath}
       \langle 0\rangle =Q_0\subsetneqq\ldots\subsetneqq Q_l=P_{R_N}
     \end{displaymath}
     of length $l+1$ in $R_N[\ux]$. Note also that by \eqref{eq:dim:1}
     \begin{equation}\label{eq:dim:2}
       l=\codim\big(P\cap L_N[\ux])=n-\dim\big(P\cap L_N[\ux]\big)=n-\dim(P),
     \end{equation}
     since $L_N[\ux]$ is a polynomial ring over a field.

     Moreover, since $1\not\in\IN_0\big(P_{R_N}\big)$ by Lemma \ref{lem:localisation}(b), there
     exists a maximal ideal $Q\lhd R_N[\ux]$ containing $P_{R_N}$ such that
     $1\not\in\IN_0(Q)$.  Choose a maximal chain of prime ideals 
     \begin{displaymath}
       P_{R_N}=Q_l\subsetneqq Q_{l+1}\subsetneqq\ldots\subsetneqq Q_k=Q
     \end{displaymath}
     in $R_N[\ux]$ from $P_{R_N}$ to $Q$, so that taking \eqref{eq:dim:2}
     into account
     \begin{equation}\label{eq:dim:3}
       \dim(P_{R_N})\geq k-l=k-n+\dim(P).
     \end{equation}

     Finally, since the sequence
     \begin{displaymath}
       \langle 0\rangle=Q_0\subsetneqq Q_1\subsetneqq\ldots\subsetneqq
       Q_l\subsetneqq\ldots\subsetneqq Q_k=Q
     \end{displaymath}
     cannot be prolonged and since $1\not\in\IN_0(Q)$, Lemma
     \ref{lem:localisation}(c) implies that $k=n+1$.  But since we always
     have
     \begin{displaymath}
       \dim\big(P_{R_N}\big)
       \leq\dim\big(R_N[\ux]\big)-\codim\big(P_{R_N}\big)
       =n+1-l,
     \end{displaymath}
     it follows from \eqref{eq:dim:2} and \eqref{eq:dim:3}
     \begin{displaymath}
       \dim(P)+1\leq \dim\big(P_{R_N}\big)\leq n+1-l=\dim(P)+1.
     \end{displaymath}

     The claim for the codimensions then follows from Lemma
     \ref{lem:localisation} (g) and (h).
   \end{proof}

   As an immediate corollary we get one of the main results of this section.

   \begin{theorem}\label{thm:dimension:A}
     Let $J\unlhd L[\ux]$ and $N\in\mathcal{N}(J)$.
     Then
     \bmath
     \dim\big(J_{R_N}\big)=\dim(J)+1
     \emath
     if and only if
     \bmath
     \exists\;P\in\Ass(J)\longer{\;}\mbox{ s.t. }\longer{\;}\dim(P)=\dim(J)
     \longer{\;\;}\mbox{ and }\longer{\;\;}1\not\in\IN_0\big(P_{R_N}\big).
     \emath
     Otherwise $\dim\big(J_{R_N}\big)=\dim(J)$.
   \end{theorem}
   \begin{proof}
     If there is such a $P\in\Ass(J)$ then
     Corollary \ref{cor:dimension:P} implies
     \begin{align*}
      & \dim\big(P_{R_N}\big)=\dim(P)+1=\dim(J)+1\mbox{ and } \\&
          \dim\big(P'_{R_N}\big)\leq\dim(P')+1\leq\dim(J)+1
     \end{align*}
     for any other $P'\in\Ass(J)$. This shows that
     \begin{displaymath}
       \dim\big(J_{R_N}\big)=
       \max\{\dim\big(P'_{R_N}\big)\;\big|\;P'\in\Ass(J)\big\}=\dim(J)+1,
     \end{displaymath}
     due to Proposition \ref{prop:pd}.

     If on the other hand $1\in\IN_0\big(P_{R_N}\big)$ for all
     $P\in\Ass(J)$ with $\dim(P)=\dim(J)$, then again by Corollary
     \ref{cor:dimension:P} $\dim(P_{R_N})\leq \dim(J)$ for all
     associated primes with equality for some, and we are done with
     Proposition \ref{prop:pd}.
   \end{proof}

   It remains to show that also the dimension of the $t$-initial ideal
   behaves well.

   \begin{proposition}\label{prop:dimformel}
     Let $I\unlhd R_N[\ux]$ be an ideal such that
     $I=I:\langle t^\frac{1}{N}\rangle^\infty$ and such that
     $1\not\in\IN_0(P)$ for some $P\in\Ass(I)$ with $\dim(P)=\dim(I)$.
     Then
     \begin{displaymath}
       \dim(I)=\dim\big(\tin_0(I)\big)+1.
     \end{displaymath}
     More precisely, $\dim(Q')=\dim(P)-1$ for all $Q'\in\minAss\big(\tin_0(P)\big)$.
   \end{proposition}
   \begin{proof}
     We first want to show that
     \begin{displaymath}
       \tin_0(I)=\big(I+\big\langle t^\frac{1}{N}\big\rangle\big)\cap K[\ux].
     \end{displaymath}
     Any element $f\in \langle t^\frac{1}{N}  \rangle +I$ can be written
     as $f=t^\frac{1}{N}\cdot g+h$ with $g\in R_N[\ux]$ and $h\in I$
     such that $\IN_0(h)\in K[\ux]$, and if in addition $f\in K[\ux]$ then
     obviously $f=\IN_0(h)=\tin_0(h)\in\tin_0(I)$. If, on the other
     hand, $g=\tin_0(f)\in\tin_0(I)$ for some $f\in I$, then
     $t^\frac{\alpha}{N}\cdot g=\IN_0(f)\in\IN_0(I)$ for some 
     $\alpha\geq 0$, and every monomial in $f$ is necessarily
     divisible by $t^\frac{\alpha}{N}$. Thus
     $f=t^\frac{\alpha}{N}\cdot h$ for some $h\in R_N[\ux]$ and
     $g=\IN_0(h)\equiv h\pmod{\langle t^\frac{1}{N}\rangle}$.
     But since 
     $I$ is saturated with respect to $t^\frac{1}{N}$ it follows that
     $h\in I$, and thus $g$ is in the right hand side. This proves the
     claim.

     Therefore, the inclusion $K[\ux]\hookrightarrow R_N[\ux]$ induces an
     isomorphism
     \begin{equation}\label{eq:dimformel:4}
       K[\ux]/\tin_0(I)\cong
       R_N[\ux]/\big(\langle t^\frac{1}{N}  \rangle +I\big)       
     \end{equation}
     which shows that
     \begin{equation}\label{eq:dimformel:0}
       \dim\big(K[\ux]/\tin_0(I)\big)       
       =
       \dim\bigg(R_N[\ux]/\Big(I+\big\langle
       t^\frac{1}{N}\big\rangle\Big)\bigg).
     \end{equation}

     Next, we want to show that 
     \begin{equation}\label{eq:dimformel:1}
       \dim\Big(P+\big\langle t^\frac{1}{N}\big\rangle\Big)=\dim(P)-1=\dim(I)-1.
     \end{equation}
     For this we consider an arbitrary $P'\in\minAss\Big(P+\big\langle
     t^\frac{1}{N}\big\rangle\Big)$. By Lemma \ref{lem:localisation} (b), $1\notin \IN_0(P')$.
     Applying  Lemma \ref{lem:localisation} (g) to $P$ and $P'$ we get
    \begin{displaymath}
       \dim(R_N[\ux])=\dim(P)+\codim(P)
       \short{\;\mbox{ and }\;}
       \longer{\end{displaymath}and\begin{displaymath}}
       \dim(R_N[\ux])=\dim(P')+\codim(P').
    \end{displaymath}
     Moreover, since $I$ is saturated with respect to
     $t^\frac{1}{N}$ by Lemma \ref{lem:saturated} $P$ does not contain
     $t^\frac{1}{N}$. Thus $t^\frac{1}{N}$
     is neither a zero divisor nor a unit in $R_N[\ux]/P$, and by Krull's Principal Ideal
     Theorem (see \cite{AM69} Cor.\ 11.17) we thus get
     \bmath
       \codim(P')=\codim(P)+1,
     \emath
     since by assumption $P'$ is minimal over $t^\frac{1}{N}$ in $R_N[\ux]/P$.
     Plugging the two previous equations in we get
     \begin{equation}\label{eq:dimformel:3}
       \dim(P')=\dim(P)-1.
     \end{equation}
     This proves \eqref{eq:dimformel:1}, since $P'$ was an arbitrary
     minimal associated prime of 
     $P+\big\langle t^\frac{1}{N}\big\rangle$.

     We now  claim that
     \begin{equation}\label{eq:dimformel:2}
       \dim\Big(P+\big\langle
       t^\frac{1}{N}\big\rangle\Big)
       =
       \dim\Big(I+\big\langle t^\frac{1}{N}\big\rangle\Big).
     \end{equation}
     Suppose this is not the case, then there is a
     $P'\in\Ass\Big(I+\big\langle t^\frac{1}{N}\big\rangle\Big)$ 
     such that 
     \begin{displaymath}
       \dim(P')>\dim\Big(P+\big\langle t^\frac{1}{N}\big\rangle\Big)=\dim(I)-1,
     \end{displaymath}
     and since $I\subset P'$ it follows that
     \begin{displaymath}
       \dim(P')=\dim(I).
     \end{displaymath}
     But then $P'$ is necessarily a minimal associated prime of $I$
     in contradiction to Lemma \ref{lem:saturated}, since $P'$
     contains $t^\frac{1}{N}$. This proves \eqref{eq:dimformel:2}.

     Equations \eqref{eq:dimformel:0}, \eqref{eq:dimformel:1} and
     \eqref{eq:dimformel:2} finish the proof of the first claim. For
     the ``more precisely'' part notice that replacing $I$ by $P$
     in \eqref{eq:dimformel:4} we see that there is a dimension
     preserving $1:1$-correspondence between $\minAss\big(P+\langle 
     t^\frac{1}{N}\rangle\big)$ and $\minAss\big(\tin_0(P)\big)$. The
     result then follows from \eqref{eq:dimformel:3}.
   \end{proof}

   \begin{remark}\label{rem:dimformel}
     The condition that $I$ is saturated with respect to
     $t^\frac{1}{N}$ in Proposition \ref{prop:dimformel} is equivalent
     to the fact that $I$ is the contraction of the ideal $\langle
     I\rangle_{L_N[\ux]}$. Moreover, it implies that $R_N[\ux]/I$ is a
     flat $R_N$-module, or alternatively that the family 
     \begin{displaymath}
       \iota^*:\Spec\big(R_N[\ux]/I\big)\longrightarrow \Spec(R_N)
     \end{displaymath}
     is flat, where the generic fibre is just
     $\Spec\big(L_N[\ux]/\langle I\rangle\big)$ and the special fibre
     is $\Spec\big(K[\ux]/\tin_0(I)\big)$. The condition $1\not\in
     \IN_0(P)$ implies that the component of
     $\Spec\big(R_N[\ux]/I\big)$ defined by $P$ surjects onto
     $\Spec(R_N)$. With this interpretation the proof of Proposition 
     \ref{prop:dimformel} is basically exploiting the dimension
     formula for local flat extensions.
   \end{remark}

   \begin{corollary}\label{cor:dimension:B}
     Let $J\lhd L[\ux]$  and $\omega\in\Q^n$, then
     \begin{displaymath}
       \dim\big(\tin_\omega(J)\big)=\max\big\{\dim(P)\;\big|\;P\in\Ass(J)\;:\;
       1\not\in\tin_\omega(P)\big\}.
     \end{displaymath}     
     Moreover, if $J$ is prime, $1\not\in\tin_\omega(J)$ and
     $Q'\in\minAss\big(\tin_\omega(J)\big)$ then 
     \begin{displaymath}
       \dim(Q')=\dim(J).
     \end{displaymath}
   \end{corollary}
   \begin{proof}
     Let
     \bmath
       J=Q_1\cap\ldots\cap Q_k
     \emath
     be a minimal primary decomposition of $J$, and 
     \begin{displaymath}
       \Phi_\omega(J)=\Phi_\omega(Q_1)\cap\ldots\cap \Phi_\omega(Q_k)
     \end{displaymath}
     the corresponding minimal primary decomposition of $\Phi_\omega(J)$. 
     If we define a new ideal
     \begin{displaymath}
       J'=\bigcap_{ 1\not\in\tin_0\big(\sqrt{\Phi_\omega(Q_i)}\big)}\Phi_\omega(Q_i),
     \end{displaymath}
     then this representation is already a minimal primary
     decomposition of $J'$.
     Choose an $N$ such that $N\in \mathcal{N}(J)$,
     $N\in\mathcal{N}(J')$ and $N\in\mathcal{N}\big(\Phi_\omega(Q_i)\big)$ for all $i=1,\ldots,k$. 
     By Lemma \ref{lem:intin} we have
     \begin{equation}\label{eq:dimin:2}
       1\not\in\tin_0\Big(\sqrt{\Phi_\omega(Q_i)}\Big)
       \;\;\;\Longleftrightarrow\;\;\;
       1\not\in\IN_0\Big(\sqrt{\Phi_\omega(Q_i)}\cap R_N[\ux]\Big).
     \end{equation}

     Proposition
     \ref{prop:pd} implies 
     \begin{displaymath}
       \Ass(J_{R_N})=\Big\{\sqrt{\Phi_\omega(Q_i)}\cap R_N[\ux]\;\Big|\;i=1,\ldots,k\Big\}
     \end{displaymath}
     where the $\sqrt{\Phi_\omega(Q_i)}\cap R_N[\ux]$ are not necessarily pairwise
     different, and 
     \begin{displaymath}
       \Ass(J'_{R_N})=\left\{\sqrt{\Phi_\omega(Q_i)}\cap
       R_N[\ux]\;\Big|\;1\not\in\IN_0\Big(\sqrt{\Phi_\omega(Q_i)}\cap R_N[\ux]\Big)\right\},
     \end{displaymath}
     for which we have to take \eqref{eq:dimin:2} into account.

     Moreover, by Lemma \ref{lem:saturation} $J'_{R_N}$ is saturated with
     respect to $t^\frac{1}{N}$. Thus we can apply
     Proposition \ref{prop:dimformel} to $J'_{R_N}$ to deduce
     \bmath
       \dim(J'_{R_N})=\dim\big(\tin_0(J'_{R_N})\big)+1.
     \emath

     Taking \eqref{eq:dimin:2} into account we can apply
     Theorem \ref{thm:dimension:A} to $J'$  and  deduce that then
     \bmath
       \dim(J'_{R_N})=\dim(J')+1,
     \emath
     but 
     \begin{align*}
       \dim(J')=&\max\big\{\dim\big(\sqrt{\Phi_\omega(Q_i)}\big)\;|
       \;1\not\in\tin_0\big(\sqrt{\Phi_\omega(Q_i)}\big)\big\}\\[0.2cm]
       =&\max\big\{\dim\big(\sqrt{Q_i}\big)\;|
       \;1\not\in\tin_\omega\big(\sqrt{Q_i}\big)\big\}.       
     \end{align*}

     It remains to show that 
     \bmath
       \tin_0(J'_{R_N})=\tin_\omega(J).
     \emath
     By Lemma \ref{lem:intin} and Definition \ref{rem:liftinglemma} we have
     \bmath
       \tin_0(J'_{R_N})=\tin_0(J')
     \emath
     and
     \begin{displaymath}
       \tin_\omega(J)=\tin_0\big(\Phi_\omega(J)\big)\subseteq\tin_0(J'),
     \end{displaymath}
     since $J\subseteq J'$. By assumption for any
     $\sqrt{\Phi_\omega(Q_i)}\not\in\Ass(J')$ there is an $f_i\in\sqrt{\Phi_\omega(Q_i)}$ such
     that $\tin_0(f_i)=1$ and there is some $m_i$ such that
     $f_i^{m_i}\in \Phi_\omega(Q_i)$. If $f\in J'$ is any element, then for
     \begin{displaymath}
       g:=f\cdot \prod_{\sqrt{\Phi_\omega(Q_i)}\not\in\Ass(J')}f_i^{m_i}\in
       \big(J'\cdot \prod_{\sqrt{\Phi_\omega(Q_i)}\not\in\Ass(J')}\Phi_\omega(Q_i)\big)
       \subseteq J
     \end{displaymath}
     we have
     \begin{displaymath}
       \tin_0(f)=\tin_0(f)\cdot \prod_{\sqrt{\Phi_\omega(Q_i)}\not\in\Ass(J')}\tin_0(f_i)^{m_i}
       =\tin_0(g)\in\tin_0(J).
     \end{displaymath}
     This finishes the proof of the first claim.

     For the ``moreover'' part note that by Lemma \ref{lem:intin}
     \begin{displaymath}
       \tin_\omega(J)=\tin_0\big(\Phi_\omega(J)\big)=\tin_0\big(\Phi_\omega(J)\cap R_N[\ux]\big)   
     \end{displaymath}
     and $\Phi_\omega(J)\cap R_N[\ux]$ is saturated and
     prime. Applying Proposition \ref{prop:dimformel}
     to 
     \begin{displaymath}
       Q'\in\minAss\Big(\tin_0\big(\Phi_\omega(J)\cap
       R_N[\ux]\big)\Big)=
       \minAss\big(\tin_\omega(J)\big)
     \end{displaymath}
     we get 
     \begin{displaymath}
       \dim(Q')=\dim\big(\Phi_\omega(J)\cap R_N[\ux]\big)-1   
       =\dim(J),
     \end{displaymath}
     where the latter equality is due to Corollary
     \ref{cor:dimension:P}. 
   \end{proof}

   \begin{theorem}\label{thm:dimension:C}
     Let $J\lhd L[\ux]$, $N\in\mathcal{N}(J)$ and
     $\omega\in\Q_{\leq 0}^n$. 

     If there is a 
     $P\in\Ass(J)$ with $\dim(P)=\dim(J)$ and
     $\omega\in\Trop(P)$, then
     \begin{displaymath}
       \dim(J_{R_N})=\dim(J)+1=\dim\big(\tin_\omega(J)\big)+1.
     \end{displaymath}
   \end{theorem}
   \begin{proof}
     By Lemma \ref{lem:T} the condition
     $\omega\in\Trop(P)\cap\Q_{\leq 0}^n$
     implies that $1\not\in\IN_0(P_{R_N})$. The result then follows from
     Theorem \ref{thm:dimension:A} and Corollary \ref{cor:dimension:B}.
   \end{proof}

   \begin{corollary}\label{cor:dimension:D}
     If $J\unlhd L[\ux]$ is zero
     dimensional  and 
     $\omega\in\Trop(J)$, then 
     \bmath
       \dim\big(\tin_\omega(J)\big)=\dim(J)=0.
     \emath
     If in addition $\Trop(J)\cap\Q_{\leq 0}^n\not=\emptyset$ and $N\in\mathcal{N}(J)$
     \bmath
       \dim\big(J_{R_N})=1.
     \emath
   \end{corollary}
   \begin{proof}
     Since $\dim(J)=0$ also $\dim(P)=0$ for every associated prime
     $P$. By \ref{lem:tropicalvariety} there exists a $P$ with
     $\omega\in \Trop(P)$. 
     The first assertion thus
     follows from Corollary \ref{cor:dimension:B}. The second assertion follows
     from Theorem \ref{thm:dimension:C}.
   \end{proof}

   When cutting down the dimension we need to understand how the
   minimal associated primes of $J$ and $J_{R_N}$ relate to each other.

   \begin{lemma}\label{lem:minAsseqdim}
     Let $J\lhd L[\ux]$ be equidimensional and $N\in\mathcal{N}(J)$. Then
     \begin{displaymath}\minAss(J_{R_N})=\{P_{R_N}\;|\;P\in\minAss(J)\}.\end{displaymath} 
   \end{lemma}
   \begin{proof}
     The left hand side is contained in the right hand side by default
     (see Proposition \ref{prop:pd}). Let therefore $P\in\minAss(J)$
     be given. By Proposition \ref{prop:pd} $P_{R_N}\in\Ass(J)$, and it
     suffices to show that it is minimal among the associated primes. 
     Suppose therefore we have
     $Q\in\Ass(J)$ such that $Q_{R_N}\subseteq P_{R_N}$. By Corollary
     \ref{cor:dimension:P} and the assumption we have
     \begin{displaymath}
       \codim(P_{R_N})=\codim(P)\leq\codim(Q)=\codim(Q_{R_N}),
     \end{displaymath}
     so that indeed $P_{R_N}=Q_{R_N}$.
   \end{proof}

   Another consequence is that the $t$-initial ideal of an
   equidimensional ideal is again equidimensional.

   \begin{corollary}\label{cor:minAsstin}
     Let $J\lhd L[\ux]$ be an equidimensional ideal and
     $\omega\in\Q^n$, then 
     \begin{displaymath}
       \minAss\big(\tin_\omega(J)\big)=\bigcup_{P\in\minAss(J)}\minAss\big(\tin_\omega(P)\big).
     \end{displaymath}
     In particular, if there is a $P\in\minAss(J)$ such that
     $1\not\in\tin_\omega(P)$ then $\tin_\omega(J)$ is equidimensional of dimension
     $\dim(J)$.
   \end{corollary}
   \begin{proof}
     Applying $\Phi_\omega$ we may assume that $\omega=0$, and we then
     may choose an $N\in \mathcal{N}(J)$ and $N\in\mathcal{N}(P)$ for all
     $P\in\minAss(J)$. 

     Denoting by
     \begin{displaymath}
       \pi:R_N[\ux]\longrightarrow R_N[\ux]/\big\langle t^\frac{1}{N}\big\rangle=K[\ux]
     \end{displaymath}
     the residue class map we get
     \begin{align*}
      & \tin_0(J)=\tin_0(J_{R_N})=\pi\big(J_{R_N}+\langle t^\frac{1}{N}\rangle\big)\mbox{ and} \\&
            \tin_0(P)=\tin_0(P_{R_N})=\pi\big(P_{R_N}+\langle t^\frac{1}{N}\rangle\big) 
     \end{align*}
     for all $P\in\minAss(J)$, where the first equality in both cases
     is due to Lemma \ref{lem:intin} and where the last equality uses
     Lemma \ref{lem:saturation}. Since there is a one-to-one
     correspondence between prime ideals in $K[\ux]$ and prime ideals
     in $R_N[\ux]$ which contain $t^\frac{1}{N}$, it suffices to show that 
     \begin{displaymath}
       \minAss\big(J_{R_N}+\langle t^\frac{1}{N}\rangle\big)
       =
       \bigcup_{P\in \minAss(J)} \minAss\big(P_{R_N}+\langle t^\frac{1}{N}\rangle\big).
     \end{displaymath}
     However, since the $P_{R_N}$ are saturated with respect to
     $t^\frac{1}{N}$ by Lemma \ref{lem:saturation} they do not
     contain $t^\frac{1}{N}$. By Corollary \ref{cor:dimension:P}
     all $P_{R_N}$ have the same codimension, since the $P$ do by
     assumption.  By Lemma \ref{lem:minAsseqdim}, 
     \begin{displaymath}
       \minAss(J_{R_N})=\{P_{R_N}\;|\;P\in\minAss(J)\}.
     \end{displaymath}
     Hence the result follows by Lemma \ref{lem:minAss}.

     The ``in particular'' part follows from Corollary \ref{cor:dimension:B}.
   \end{proof}


   \section{Computing $t$-Initial Ideals}\label{sec:computinginitialideals}

   This section is devoted to an alternative proof of Theorem
   \ref{thm:stdtin} which does not need standard basis in the mixed
   power series polynomial ring $K[[t]][\ux]$.

   The following lemma is easy to show.

   \begin{lemma}\label{lem:initialform}
     Let $w\in\R_{<0}\times\R^n$, $0\not=f=\sum_{i=1}^k g_i\cdot
     h_i$ with $f,g_i,h_i\in R_N[\ux]$ and 
     $\ord_w(f)\geq \ord_w(g_i\cdot h_i)$ for all $i=1,\ldots,k$. Then
     \begin{displaymath}
       \IN_w(f)\in\big\langle \IN_w(g_1),\ldots,\IN_w(g_k)\big\rangle\lhd K\big[t^\frac{1}{N},\ux\big].
     \end{displaymath}
   \end{lemma}
   \longer{\begin{proof}
     Due to the direct product decomposition in Definition
     \ref{def:initial} we have that 
     \begin{displaymath}
       \IN_w(f)=f_{\hat{q},w}=\sum_{i=1}^k (g_i\cdot h_i)_{\hat{q},w}
     \end{displaymath}
     where $\hat{q}=\ord_w(f)$. By assumption $\ord_w(g_i)+\ord_w(h_i)=\ord_w(g_i\cdot
     h_i)\leq\ord_w(f)=\hat{q}$ with equality if
     and only if $(g_i\cdot h_i)_{\hat{q},w}\not=0$. In that case
     necessarily 
     \begin{displaymath}
       (g_i\cdot h_i)_{\hat{q},w}=\IN_w(g_i)\cdot \IN_w(h_i),
     \end{displaymath}
     which finishes the proof.
   \end{proof}
   }

  \begin{proposition}\label{prop:tinstd}
    Let $I\unlhd K\big[t^\frac{1}{N},\ux\big]$, 
    $\omega\in\Q^n$ and $G$ be a standard basis of
    $I$ with respect to the monomial ordering $>_\omega$ introduced in Remark
    \ref{rem:monomialordering}.
    Then
    \begin{displaymath}
      \IN_\omega(I)=\big\langle\IN_\omega(G)\big\rangle\unlhd
      K\big[t^\frac{1}{N},\ux\big] \short{\;\mbox{ and }\;}
      \longer{\end{displaymath}and\begin{displaymath}}
      \tin_\omega(I)=\big\langle\tin_\omega(G)\big\rangle\unlhd K[\ux].
    \end{displaymath}
  \end{proposition}
  \begin{proof}
    It suffices to show that 
    \bmath
    \IN_\omega(f)\in\langle\IN_\omega(G)\rangle
    \emath
    for every $f\in I$. Since $f\in I$ and $G$ is a standard basis of
    $I$ there exists a weak standard representation 
    \begin{equation}\label{eq:tinstd:1}
      u\cdot f=\sum_{g\in G} q_g\cdot g
    \end{equation}
    of $f$ where the leading term of $u$  with
    respect to $>_\omega$ is $\lt_{>_\omega}(u)=1$. But then the
    definition of $>_\omega$ implies that automatically
    $\IN_\omega(u)=1$. Since \eqref{eq:tinstd:1} is a standard
    representation we have $\lm_{>_\omega}(u\cdot f)\geq
    \lm_{>_\omega}(q_g\cdot g)$ for all $g$. But this necessarily
    implies that $\ord_w(f)\geq \ord_w(q_g\cdot g)$ where
    $w=(-1,\omega)$. Since $K\big[t^\frac{1}{N},\ux\big]\subset
    R_N[\ux]$ we can use Lemma \ref{lem:initialform} to show
    \begin{displaymath}
      \IN_w(f)=\IN_w(u\cdot f)\in\langle \IN_w(g)\;|\;g\in
      G\rangle\unlhd{K\big[t^\frac{1}{N},\ux\big]}. 
    \end{displaymath}
  \end{proof}
  
  \begin{proposition}
    Let $I\subseteq K[t,x]$ be an ideal, $J=\langle I\rangle_{L[x]}$
    and $\omega\in\R^n$. Then
    $\textup{t-in}_\omega(I)=\textup{t-in}_\omega(J)$. 
  \end{proposition}
  \begin{proof}
    We need to prove the inclusion
    $\textup{t-in}_\omega(I)\supseteq\textup{t-in}_\omega(J)$. The other
    inclusion is clear since $I\subseteq J$.  The right hand side is
    generated by elements of the form $f=\textup{t-in}_\omega(g)$ where
    $g\in J$. Consider such $f$ and $g$. The polynomial $g$ must be of the
    form $g=\sum_i c_i\cdot g_i$ where $g_i\in I$ and
    $c_i\in L$. Let $d$ be the $(-1,\omega)$-degree of
    $\textup{in}_\omega(g)$. The degrees of terms in $g_i$ are
    bounded. Terms $a\cdot t^\beta$ in $c_i$ of large enough $t$-degree
    will make the $(-1,\omega)$-degree of $a\cdot t^\beta\cdot g_i$ drop below
    $d$ since the degree of $t$ is negative. Consequently, these terms can
    simply be ignored since they cannot affect the initial form of
    $g=\sum_i c_i\cdot g_i$. Renaming and possibly repeating some
    $g_i$'s we may write $g$ as a finite sum $g=\sum_i c'_i\cdot g_i$
    where 
    $c'_i=a_i\cdot t^{\beta_i}$ and $g_i\in I$ with $a_i\in K$ and
    $\beta_i\in \Q$.  We will split the sum into subsums grouping together
    the $c'_i$'s that have the same $t$-exponent modulo $\Z$. For suitable
    index sets $A_j$ we let $g=\sum_j G_j$ where $G_j=\sum_{i\in A_j}c'_i\cdot
    g_i$. Notice that all $t$-exponents in a $G_j$ are congruent modulo
    $\Z$ while $t$-exponents from different $G_j$'s are not. In particular
    there is no cancellation in the sum $g=\sum_j G_j$. As a consequence
    $\textup{in}_\omega(g)=\sum_{j\in S}\textup{in}_\omega(G_j)$ for a
    suitable subset $S$. We also have
    $\textup{t-in}_\omega(g)=\sum_{j\in S}\textup{t-in}_\omega(G_j)$. We
    wish to show that each $\textup{t-in}_\omega(G_j)$ is in
    $\textup{t-in}(I)$. We can write $t^{\gamma_j}\cdot G_j=\sum_{i\in
      A_j}t^{\gamma_j}\cdot c'_i\cdot g_i$ for suitable $\gamma_j\in\Q$ such that
    $t^{\gamma_j}\cdot c_i'\in K[t]$ for all $i\in A_j$. Observe that
    $$\textup{t-in}_\omega(G_j)=\textup{t-in}_\omega(t^{\gamma_j}\cdot
    G_j)=\textup{t-in}_\omega\Big(\sum_{i\in 
      A_j}t^{\gamma_j}\cdot c'_i\cdot g_i\Big)\in \textup{t-in}_\omega(I).$$ Applying
    $\textup{t-in}_\omega(g)=\sum_{j\in S}\textup{t-in}_\omega(G_j)$ we
    see that $f=\textup{t-in}_\omega(g)\in\textup{t-in}_\omega(I)$.
  \end{proof}

  By substituting $t:=t^{\frac{1}{n}}$ and scaling $\omega$ we get
  Theorem \ref{thm:stdtin} as a corollary.

\newcommand{\etalchar}[1]{$^{#1}$}
\providecommand{\bysame}{\leavevmode\hbox to3em{\hrulefill}\thinspace}
\providecommand{\MR}{\relax\ifhmode\unskip\space\fi MR }
\providecommand{\MRhref}[2]{%
  \href{http://www.ams.org/mathscinet-getitem?mr=#1}{#2}
}
\providecommand{\href}[2]{#2}

\end{document}